\documentclass[10pt]{article}
\usepackage[margin=1in]{geometry}
\usepackage[utf8]{inputenc}

\usepackage{fullpage}
\usepackage{enumerate}
\usepackage{authblk}
\usepackage[colorlinks=true, citecolor=ForestGreen, linkcolor=BrickRed, urlcolor=blue]{hyperref}
\usepackage{amsthm,amsmath, amssymb, amsfonts,multicol}
\usepackage{authblk}
\usepackage{verbatim,lineno}
\usepackage{tikz}
\usetikzlibrary{fit,matrix,positioning,shapes}
\usetikzlibrary{decorations.shapes}
\usepackage{graphicx,float}
\usepackage{mathtools}
\usepackage[dvipsnames]{xcolor}

\usepackage{mathrsfs}

\newcommand{\precdot}{\prec\mathrel{\mkern-5mu}\mathrel{\cdot}}

\newtheorem{theorem}{Theorem}[section]
\newtheorem{lemma}[theorem]{Lemma}
\newtheorem{corollary}[theorem]{Corollary}

\newtheorem{ex}[theorem]{Example}

\newtheorem{prop}[theorem]{Proposition}
\newtheorem*{theorem*}{Theorem}
\newtheorem{remark}[theorem]{Remark}
\newtheorem*{que*}{Question}

\newcommand{\ind}{{\rm ind~}}

\newcommand{\demph}[1]{\textcolor{RoyalBlue}{\emph{#1}}}

\begin{document}

\title{The nilradical of a seaweed algebra}

\author[*]{Vincent E. Coll}

\author[**]{Nicholas Mayers}

\affil[*]{Department of Mathematics, Lehigh University, Bethlehem, PA, 18015}

\affil[**]{Department of Mathematics, Kennesaw State University, Kennesaw, GA, 30144}

%\linenumbers
\maketitle

\begin{abstract} 
\noindent
Seaweed subalgebras of $\mathfrak{gl}(n,\mathbb{C})$ and $\mathfrak{sl}(n,\mathbb{C})$ are combinatorially defined matrix Lie algebras whose index admits a closed-form description in terms of an associated graph called a meander. In this paper, we study the nilradicals of these algebras with our main result establishing an explicit formula for their index in terms of an edge-weighted variation of the meander. We further prove that each such nilradical decomposes as a direct sum of the center of the seaweed subalgebra with a nilpotent Lie poset algebra, and we provide a meander-theoretic procedure for recovering the underlying poset.
\end{abstract}


\section{Introduction}

Seaweed subalgebras of reductive Lie algebras arise as intersections of pairs of opposite parabolic subalgebras. In the case of seaweed subalgebras of $\mathfrak{gl}(n,\mathbb{C})$ and $\mathfrak{sl}(n,\mathbb{C})$, there exist explicit combinatorial descriptions in terms of pairs of compositions (see Section~\ref{sec:prelim}).  Since their introduction by Dergachev and Kirillov \cite{DK}, seaweed subalgebras have been the focus of sustained study, in part because they form a class of Lie algebras for which certain structural invariants admit concrete computation.  These computations are often carried out using combinatorial data encoded in an associated planar graph, called a \textit{meander}
\cite{Coll1,CHM,DK,Joseph,Panyushev1,Panyushev2}.

One such invariant is the \textit{index}, which in the case of seaweed subalgebras of $\mathfrak{gl}(n,\mathbb{C})$ admits a closed-form description in terms of the associated meander. The \demph{index} of a Lie algebra $\mathfrak{g}$, originally defined by Dixmier \cite{D}, is given by
\[
\ind\mathfrak{g}
=\min_{F\in \mathfrak{g}^*}\dim\ker F([-,-]).
\]
Letting $\mathfrak{s}$ be a seaweed subalgebra of $\mathfrak{gl}(n,\mathbb{C})$ and $\mathrm{M}(\mathfrak{s})$ denote its associated meander, Dergachev and Kirillov established the following result \cite{DK}.

\begin{theorem}\label{thm:indexform}
If $\mathfrak{s}$ is a seaweed subalgebra of $\mathfrak{gl}(n,\mathbb{C})$, then

\begin{equation}\label{one}
\ind \mathfrak{s} = 2C + P,
\end{equation}
where $C$ is the number of cycles and $P$ is the number of paths in $\mathrm{M}(\mathfrak{s})$.
\end{theorem}

\noindent
From this result, one obtains a corresponding formula for the index of seaweed subalgebras of $\mathfrak{sl}(n,\mathbb{C})$ by subtracting one from the right-hand side of ($\ref{one}$). Owing to its connections with areas such as deformation theory and quantum group theory \cite{G1,G2}, the problem of computing the  index of Lie algebras has received considerable attention \cite{indexnil,Cameron,CHM,CM,indexnilposet,tBCD,DK,D,Elash,onindexnil,Panyushev1,Panyushev2,Panyushev3,TY}.

In this paper, we study the index of nilradicals of seaweed subalgebras of $\mathfrak{gl}(n,\mathbb{C})$ and $\mathfrak{sl}(n,\mathbb{C})$.\footnote{A formula for the index of nilradicals of certain seaweed subalgebras can be found in \cite{TY}. The formula of \cite{TY} is given in terms of algebraic invariants of the algebra, while that developed here is in terms of an edge-weighted graph. It is worth mentioning that our proofs do not depend on the formula of \cite{TY}, but instead on a relationship between the algebras of interest here and those of \cite{indexnilposet}.}  Our first step is to show that the nilradical of such a seaweed subalgebra $\mathfrak{s}$ decomposes as a direct sum of the center of $\mathfrak{s}$ with a nilpotent Lie algebra $\mathfrak{g}^{\prec}(\mathcal{P}_{\mathfrak{s}})$ determined by a poset $\mathcal{P}_{\mathfrak{s}}$.  More precisely, the algebra $\mathfrak{g}^{\prec}(\mathcal{P}_{\mathfrak{s}})$ is a nilpotent Lie poset algebra in the sense of \cite{indexnilposet}, and we show that the poset $\mathcal{P}_{\mathfrak{s}}$ can
be recovered from a diagram closely related to $\mathrm{M}(\mathfrak{s})$.

With respect to the index, a key ingredient in our approach is a poset-theoretic formula for the index of nilpotent Lie poset algebras established in \cite{indexnilposet}.  By combining this result with the decomposition discussed above, we obtain a combinatorial formula for the index of the nilradical of $\mathfrak{s}$ in terms of an edge-weighted modification of the meander $\mathrm{M}(\mathfrak{s})$.  Concretely, the index is obtained by summing the edge weights together with the number of certain subgraphs, called \textit{central components}, of the modified meander, subtracting one in the case of seaweed subalgebras of $\mathfrak{sl}(n,\mathbb{C})$.  In addition, we show that the number of edges in the original meander $\mathrm{M}(\mathfrak{s})$, together with the number of central components (again with the appropriate subtraction in the $\mathfrak{sl}(n,\mathbb{C})$ case), yields a lower bound for the index of the nilradical which is sharp for certain families.

The remainder of the paper is organized as follows.  In Section~\ref{sec:prelim}, we review background material on seaweed subalgebras, posets, and nilpotent Lie poset algebras. In Section~\ref{sec:posetalg}, we show that the nilradicals of seaweed subalgebras of $\mathfrak{gl}(n,\mathbb{C})$ and $\mathfrak{sl}(n,\mathbb{C})$ decompose as the direct sum of the center of the seaweed subalgebra with a nilpotent Lie poset algebra, and we describe how the associated poset is constructed.  In Section~\ref{sec:index}, we use this decomposition to derive a combinatorial formula for the index of the nilradical in terms of an edge-weighted meander. Finally, in Section~\ref{sec:fd}, we outline several directions for future work.

\section{Preliminaries}\label{sec:prelim}

In this section we introduce the necessary background concerning seaweeed subalgebras, posets, and nilpotent Lie poset algebras. For $n\in\mathbb{Z}_{>0}$, we let $[n]:=\{1,\hdots,n\}$.

\subsection{Seaweed algebras}\label{sec:prelimsw}

Given a reductive Lie algebra $\mathfrak{r},$ a \demph{seaweed subalgebra of} $\mathfrak{r}$ is the intersection of two parabolic subalgebras $\mathfrak{p},\mathfrak{p}'\subset\mathfrak{r}$ satisfying $\mathfrak{p}+\mathfrak{p}'=\mathfrak{r}.$ When restricting to seaweed subalgebras of $\mathfrak{gl}(N,\mathbb{C})$ and $\mathfrak{sl}(N,\mathbb{C})$, equivalent definitions in terms of compositions are available which prove more useful to us here. For brevity, ongoing we refer to seaweed subalgebras of $\mathfrak{gl}(N,\mathbb{C})$ simply as ``seaweed algebras", and seaweed subalgebras of $\mathfrak{sl}(N,\mathbb{C})$ as ``type-A seaweed algebras".

Let $E^N_{i,j}$ denote the $N\times N$ matrix with a 1 in the $(i,j)$-entry and $0$'s elsewhere. When the size $N$ of the matrix is clear, we omit $N$ from the notation. Recall that a \demph{composition} of a positive integer $N$ is an ordered list $(a_1,\hdots,a_n)\in\mathbb{Z}^n_{>0}$ for $n\ge 1$ with $\sum_{i=1}^na_i=N$. For a composition $\mathbf{a}=(a_1,\hdots,a_n)$ of $N$, define the ordered set partition $S(\mathbf{a}):=A_1\cup\cdots\cup A_n$ where $$A_1=\{1,\hdots,a_1\}\quad\text{and}\quad A_i=\left\{\sum_{j=1}^{i-1}a_j+1,\hdots,\sum_{j=1}^{i}a_j\right\}\quad~\text{for}~1<i\le n.$$ Given two compositions of $N$, say $\mathbf{a}=(a_1,\hdots,a_n)$ and $\mathbf{b}=(b_1,\hdots,b_m),$ if $$S(\mathbf{a})=A_1\cup\cdots\cup A_n\quad\text{and}\quad S(\mathbf{b})=B_1\cup\cdots\cup B_m,$$ then define the \demph{seaweed algebra} $\mathfrak{p}\frac{a_1|\cdots|a_n}{b_1|\cdots|b_m}$ to be the subalgebra of $\mathfrak{gl}(N,\mathbb{C})$ whose underlying vector space is spanned by the collection $$\{E_{j,j}~|~j\in [N]\}\cup \bigcup_{i=1}^n\left\{E_{p,q}~|~p,q\in A_i~\text{with}~p>q\right\}\cup \bigcup_{i=1}^m\left\{E_{p,q}~|~p,q\in B_i~\text{with}~p<q\right\}.$$ When the entries of the compositions $\mathbf{a}$ and $\mathbf{b}$ are not important, we simply denote the seaweed algebra $\mathfrak{p}\frac{a_1|\cdots|a_n}{b_1|\cdots|b_m}$ by $\mathfrak{p}\frac{\mathbf{a}}{\mathbf{b}}$. Similarly, define the \demph{type-A seaweed algebra} $\mathfrak{p}^A\frac{a_1|\cdots|a_n}{b_1|\cdots|b_m}$, or simply $\mathfrak{p}^A\frac{\mathbf{a}}{\mathbf{b}}$, to be the subalgebra of $\mathfrak{sl}(N,\mathbb{C})$ whose underlying vector space is spanned by the collection $$\{E_{1,1}-E_{j,j}~|~j\in [N]\backslash\{1\}\}\cup \bigcup_{i=1}^n\left\{E_{p,q}~|~p,q\in A_i~\text{with}~p>q\right\}\cup \bigcup_{i=1}^m\left\{E_{p,q}~|~p,q\in B_i~\text{with}~p<q\right\}.$$ For an example, the seaweed algebra $\mathfrak{p}\frac{2|4}{1|2|3}$ consists of all matrices in $\mathfrak{gl}(6,\mathbb{C})$ of the form illustrated in Figure~\ref{fig:seaweed} (left) where $\ast$'s denote potential nonzero entries. Restricting to those matrices with trace zero results in the type-A seaweed algebra $\mathfrak{p}^A\frac{2|4}{1|2|3}$.

\begin{figure}[H]
$$\begin{tikzpicture}

\node at (0,0) {$\begin{bmatrix}
    * & & & & & \\
    * & * & * & & & \\
    & & * & & &  \\
    & & * & * & * & * \\
    & & * & * & * & * \\
    & & * & * & * & * \\
\end{bmatrix}$};
\draw[thick] (1.45,-1.25)--(1.45,0)--(-0.1,0)--(-0.1,0.85)--(-1.15,0.85)--(-1.15,1.25)--(-1.5,1.25);
\draw[thick] (-1.45,1.25)--(-1.45,0.45)--(-0.5,0.45)--(-0.5,-1.25)--(1.45,-1.25);
\node at (-1.7,0.9) {$2$};
\node at (-0.7,-0.5) {$4$};
\node at (-1.3,1.5) {$1$};
\node at (-0.6,1.1) {$2$};
\node at (0.7,0.2) {$3$};

\end{tikzpicture}\hspace{1.5cm}\begin{tikzpicture}[scale=1.3]
	\def\Node{\node [circle,  fill, inner sep=2pt]}
     \Node[label=left:$v_1$] (1) at (0,1.8) {};
	\Node[label=left:$v_2$] (2) at (1,1.8) {};
	\Node[label=left:$v_3$] (3) at (2,1.8) {};
	\Node[label=left:$v_4$] (4) at (3,1.8) {};
	\Node[label=left:$v_5$] (5) at (4,1.8) {};
	\Node[label=left:$v_6$] (6) at (5,1.8) {};
	\draw (1) to[bend left=50] (2);
	\draw (3) to[bend left=50] (6);
	\draw (4) to[bend left=50] (5);
	\draw (2) to[bend right=50] (3);
	\draw (4) to[bend right=50] (6);
\end{tikzpicture}$$
\caption{The seaweed algebra $\mathfrak{p}\frac{2|4}{1|2|3}$ (left) and its associated meander (right)}\label{fig:seaweed}
\end{figure}
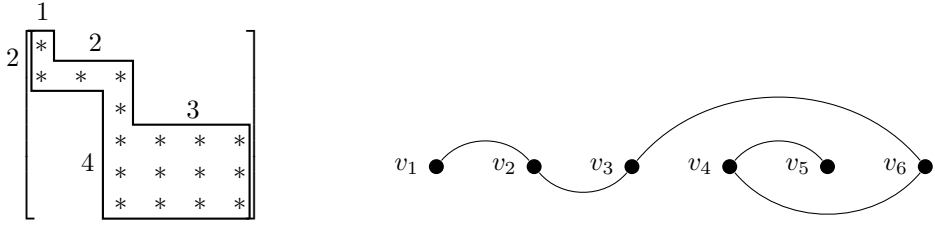

As noted in the introduction, Dergachev and Kirillov \cite{DK} assign to each seaweed algebra $\mathfrak{s}=\mathfrak{p}\frac{a_1|\cdots|a_n}{b_1|\cdots|b_m}$ with $\sum_{i=1}^na_i=\sum_{j=1}^mb_j=N$ a planar graph $\mathrm{M}(\mathfrak{s})$ on $N$ vertices, called a \demph{meander}, which can be used to compute the index of $\mathfrak{s}$ via Theorem~\ref{thm:indexform}. Since we will extend the meander $\mathrm{M}(\mathfrak{s})$ of Dergachev and Kirillov in Section~\ref{sec:index}, we recall the definition of $\mathrm{M}(\mathfrak{s})$ here. Begin by placing $N$ labeled vertices $v_1, v_2, \dots , v_N$ from left to right along a horizontal line. Next, partition the vertices by grouping together the first $a_1$ vertices, then the next $a_2$ vertices, and so on, lastly grouping together the final $a_m$ vertices. We call each set of vertices formed a \demph{top block}. For each top block, add a concave-down edge, called a \demph{top edge}, from the first vertex of the block to the last vertex of the block, then add a top edge between the second vertex of the block and the second-to-last vertex of the block, and so on within each block, assuming that the vertices being connected are distinct. In a similar fashion, partition the vertices according to the composition $(b_1,\hdots,b_m)$ into \demph{bottom blocks}, and then place \demph{bottom edges}, i.e., concave-up edges, between vertices in each bottom block as in the case of top blocks. See Figure \ref{fig:seaweed} (right) for an illustration of $\mathrm{M}(\mathfrak{s})$ when $\mathfrak{s}=\mathfrak{p}\frac{2|4}{1|2|3}$. Moving forward, we will omit vertex labels in illustrations of $\mathrm{M}(\mathfrak{s})$ for the sake of clarity. Note that meander graphs can contain \demph{parallel edges}, i.e., pairs of edges between the same vertices. Since it will prove helpful later, we denote the collection of edges of $\mathrm{M}(\mathfrak{s})$ excluding parallel edges by $\mathrm{E}_1(\mathfrak{s})$.

Now, in extending $\mathrm{M}(\mathfrak{s})$, our goal is to establish a meander-theoretic formula for the index of the nilradicals of (type-A) seaweed algebras. Recall that an \demph{ideal} of a Lie algebra $\mathfrak{g}$ is a subspace $\mathfrak{k}$ of $\mathfrak{g}$ satisfying $[x,y]\in\mathfrak{k}$ for all $x\in\mathfrak{g}$ and $y\in\mathfrak{k}$. Letting $\mathfrak{g}_0=\mathfrak{g}$ and $$\mathfrak{g}_n=[\mathfrak{g},\mathfrak{g}_{n-1}]=\{[x,y]~|~x\in\mathfrak{g},~y\in\mathfrak{g}_{n-1}\}\quad\text{for}\quad n>0,$$ we refer to the sequence $\mathfrak{g}_i$ for $i\ge 0$ as the \demph{lower central series} of $\mathfrak{g}$, and say that $\mathfrak{g}$ is \demph{nilpotent} when there exists $m>0$ for which $\mathfrak{g}_m=\{0\}$. The \demph{nilradical} of a Lie algebra $\mathfrak{g}$, denoted $\mathfrak{n}(\mathfrak{g})$, is the unique largest (with respect to inclusion) nilpotent ideal of $\mathfrak{g}$. Contained in the nilradical of any Lie algebra $\mathfrak{g}$ is its \demph{center}, denoted $\mathrm{Z}(\mathfrak{g})$ and defined by $$\mathrm{Z}(\mathfrak{g}):=\{x\in\mathfrak{g}~|~[x,y]=0~\text{for all}~y\in\mathfrak{g}\}.$$

To aid in identifying the nilradical of a (type-A) seaweed subalgebra $\mathfrak{s}$ in Section~\ref{sec:posetalg}, we define special subgraphs of $\mathrm{M}(\mathfrak{s})$. First, for a composition $\mathbf{c}=(c_1,\hdots,c_n)$, define $$\mathrm{PS}(\mathbf{c}):=\left\{\sum_{i=1}^jc_i~|~1\le j\le n\right\},$$ i.e., $\mathrm{PS}(\mathbf{c})$ is the set of partial sums of the entries of $\mathbf{c}$. Then for either $\mathfrak{s}=\mathfrak{p}\frac{\mathbf{a}}{\mathbf{b}}$ or $\mathfrak{p}^A\frac{\mathbf{a}}{\mathbf{b}}$, if $$\mathrm{PS}(\mathbf{a})\cap\mathrm{PS}(\mathbf{b})=\{\lambda_1<\lambda_2<\cdots<\lambda_\ell\},$$ then define the \demph{central components} of $\mathrm{M}(\mathfrak{s})$ to be the subgraphs of $\mathrm{M}(\mathfrak{s})$ corresponding to the collections of vertices $\{v_i~|~1\le i\le \lambda_1\}$ and $\{v_i~|~\lambda_{j-1}<i\le\lambda_j\}$ for $1<i\le \ell$. We denote the collection of central components of $\mathrm{M}(\mathfrak{s})$ by $\mathrm{Cen}(\mathfrak{s})$. Note that $|\mathrm{Cen}(\mathfrak{s})|\ge 1$ as $\mathrm{PS}(\mathbf{a})\cap\mathrm{PS}(\mathbf{b})$ always contains the common sum of the entries of $\mathbf{a}$ and $\mathbf{b}$. Moreover, considering the definition of $\mathrm{M}(\mathfrak{s})$, if $v_i$ and $v_j$ are vertices of $\mathrm{M}(\mathfrak{s})$ which are connected by an edge, then they belong to the same central component. The following lemma includes some properties of $\mathrm{Cen}(\mathfrak{s})$ which will be used later.

\begin{lemma}\label{lem:cc}
Let $\mathfrak{s}$ be a (type-A) seaweed algebra.
\begin{enumerate}
    \item[$(a)$] If $E_{i,j}\in\mathfrak{s}$, then $v_i$ and $v_j$ are contained in a central component of $\mathrm{M}(\mathfrak{s})$.
    \item[$(b)$] If $v_i$ and $v_j$ with $i<j$ are contained in a central component of $\mathrm{M}(\mathfrak{s})$ and $i\le k<j$, then either $E_{k,k+1}\in\mathfrak{s}$ or $E_{k+1,k}\in\mathfrak{s}$.
\end{enumerate}
\end{lemma}
\begin{proof}
    Part (a) follows immediately from the definitions of $\mathfrak{s}$, $\mathrm{M}(\mathfrak{s})$, and central components of $\mathrm{M}(\mathfrak{s})$. As for (b), assume otherwise. That is, assume that there exists $i\le k<j$ such that neither $E_{k,k+1}$ nor $E_{k+1,k}$ belong to $\mathfrak{s}$. Then, assuming $\mathfrak{s}=\mathfrak{p}\frac{\mathbf{a}}{\mathbf{b}}$ or $\mathfrak{p}^A\frac{\mathbf{a}}{\mathbf{b}}$ for compositions $$\mathbf{a}=(a_1,\hdots,a_n)\qquad\text{and}\qquad\mathbf{b}=(b_1,\hdots,b_m)$$ with $$S(\mathbf{a})=A_1\cup\cdots\cup A_n\qquad\text{and}\qquad S(\mathbf{b})=B_1\cup\cdots\cup B_m,$$ it follows that there exists $k_1\in [n]$ and $k_2\in [m]$ for which $k\in A_{k_1}\cap B_{k_2}$ and $k+1\in A_{k_1+1}\cap B_{k_2+1}$; but then $$i\le \sum_{\ell=1}^{k_1}a_\ell=\sum_{\ell=1}^{k_2}b_\ell=k<j$$ so that $k\in\mathrm{PS}(\mathbf{a})\cap\mathrm{PS}(\mathbf{b})$ and the vertices $v_i$ and $v_j$ cannot belong to the same central component of $\mathrm{M}(\mathfrak{s})$.
\end{proof}

\begin{ex}
    For $\mathfrak{s}=\mathfrak{p}\frac{2|2|3|1|1|3}{4|3|5}$, the set $\mathrm{Cen}(\mathfrak{s})$ consists of the subgraphs of $\mathrm{M}(\mathfrak{s})$ induced by the following collections of vertices: $\{v_1,v_2,v_3,v_4\}$, $\{v_5,v_6,v_7\}$, and $\{v_8,v_9,v_{10},v_{11},v_{12}\}$. The meander $\mathrm{M}(\mathfrak{s})$ is illustrated in Figure~\ref{fig:cencomp} with the central components identified.
    \begin{figure}[h]
        $$\begin{tikzpicture}[scale=1.3]
	\def\Node{\node [circle,  fill, inner sep=2pt]}
	\node at (0,0) {};
     \Node (1) at (0,0) {};
	\Node (2) at (1,0) {};
	\Node (3) at (2,0) {};
	\Node (4) at (3,0) {};
	\Node (5) at (4,0) {};
	\Node (6) at (5,0) {};
    \Node (7) at (6,0) {};
	\Node (8) at (7,0) {};
	\Node (9) at (8,0) {};
	\Node (10) at (9,0) {};
    \Node (11) at (10,0) {};
    \Node (12) at (11,0) {};
    \draw (1) to[bend left=50] (2);
    \draw (3) to[bend left=50] (4);
    \draw (5) to[bend left=50] (7);
    \draw (10) to[bend left=50] (12);
    \draw (1) to[bend right=50] (4);
    \draw (2) to[bend right=50] (3);
    \draw (5) to[bend right=50] (7);
    \draw (8) to[bend right=50] (12);
    \draw (9) to[bend right=50] (11);
    \draw (-0.25,-1)--(-0.25,-1.25)--(3.25,-1.25)--(3.25,-1);
    \draw (3.75,-1)--(3.75,-1.25)--(6.25,-1.25)--(6.25,-1);
    \draw (6.75,-1)--(6.75,-1.25)--(11.25,-1.25)--(11.25,-1);
    \node at (1.5, -1.5) {Central Component 1};
    \node at (5, -1.5) {Central Component 2};
    \node at (9, -1.5) {Central Component 3};
\end{tikzpicture}$$
\caption{$\mathrm{M}(\mathfrak{s})$ for $\mathfrak{s}=\mathfrak{p}\frac{2|2|3|1|1|3}{4|3|5}$}\label{fig:cencomp}
    \end{figure}
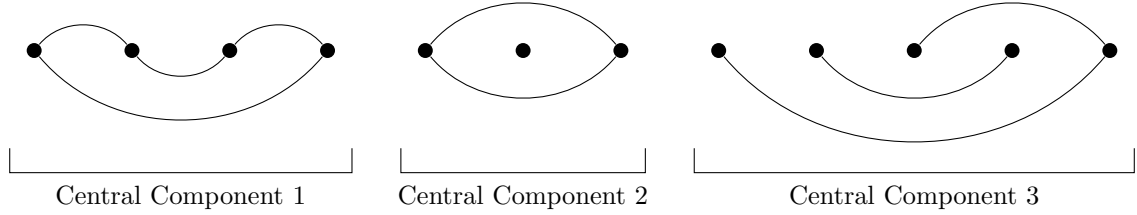
\end{ex}

Comparing the definition of $\mathrm{M}(\mathfrak{s})$ to that of its central components, we have that following.

\begin{prop}\label{prop:cccom}
    If $\mathfrak{s}$ is a (type-A) seaweed algebra, then $|\mathrm{Cen}(\mathfrak{s})|$ is equal to one more than the number of pairs of consecutive vertices $v_i$ and $v_{i+1}$ in $\mathrm{M}(\mathfrak{s})$ for which no edges cross over the region between the vertices.
\end{prop}

\subsection{Posets and Nilpotent Lie poset algebras}

A \demph{finite poset} $(\mathcal{P}, \preceq_{\mathcal{P}})$ consists of a finite set $\mathcal{P}=[n]$ together with a binary relation $\preceq_{\mathcal{P}}$ which is reflexive, anti-symmetric, and transitive. Ongoing, let $\le$ denote the natural ordering on elements of $\mathbb{Z}$ and, when no confusion will arise, we simply denote a poset $(\mathcal{P}, \preceq_{\mathcal{P}})$ by $\mathcal{P}$, and $\preceq_{\mathcal{P}}$ by $\preceq$. Two posets $\mathcal{P}$ and $\mathcal{Q}$ are \demph{isomorphic}, denoted $\mathcal{P}\cong \mathcal{Q}$, if there exists an order-preserving bijection between their underlying sets whose inverse is also order-preserving.

Let $p_1,p_2\in\mathcal{P}$. If $p_1\preceq p_2$ and $p_1\neq p_2$, then we call $p_1\preceq p_2$ a \demph{strict relation} and write $p_1\prec p_2$. An element $p\in\mathcal{P}$ is called \demph{minimal} (resp., \demph{maximal}) if there exists no $q\in\mathcal{P}$ for which $q\prec p$ (resp., $p\prec q$). We denote by $\mathrm{Rel}(\mathcal{P})$ the set of strict relations between elements of $\mathcal{P}$, by $\mathrm{Min}(\mathcal{P})$ the set of minimal elements of $\mathcal{P}$, by $\mathrm{Max}(\mathcal{P})$ the set of maximal elements of $\mathcal{P}$, and by $\mathrm{Ext}(\mathcal{P})$ the set $\mathrm{Min}(\mathcal{P})\cup\mathrm{Max}(\mathcal{P})$. For examples of these collections, see Example~\ref{ex:posnot}. If $p_1\prec p_2$ and there does not exist $p\in \mathcal{P}$ satisfying $p_1\prec p\prec p_2$, then $p_1\prec p_2$ is a \demph{covering relation} and we write $p_1\precdot p_2$.  The \demph{Hasse diagram} of $\mathcal{P}$ is the graph whose vertices are the elements of $\mathcal{P}$ where if $p\prec q$ in $\mathcal{P}$ then the vertex $q$ is drawn above $p$ (i.e., with larger $y$ coordinate), and vertices $p$ and $q$ are connected by an edge when $p\precdot q$ (see, for example, Figure~\ref{fig:poset}). We refer to the maximal (with respect to containment) connected subposets of $\mathcal{P}$ as the \demph{connected components} of $\mathcal{P}$.  

\begin{ex}\label{ex:posnot}
Let $\mathcal{P}$ be the poset $\mathcal{P}=[6]$ with $1,2\prec 3\prec 4,5,6$. In Figure~\ref{fig:poset} (left) we illustrate the Hasse diagram of $\mathcal{P}$. We have $$\mathrm{Rel}(\mathcal{P})=\{1\prec 3,1\prec 4,1\prec 5,1\prec 6,2\prec 3,2\prec 4,2\prec 5,2\prec 6,3\prec 4,3\prec 5,3\prec 6\},\quad \mathrm{Min}(\mathcal{P})=\{1,2\},$$ $$\mathrm{Max}(\mathcal{P})=\{4,5,6\},\quad \text{and}\quad \mathrm{Ext}(\mathcal{P})=\{1,2,4,5,6\}.$$

\begin{figure}[H]
$$\begin{tikzpicture}[scale=0.8]
	\node [circle, draw = black, fill = black, inner sep = 0.5mm, label=below:{1}] (1) at (-0.5,0) {};
	\node (2) at (0, 1)[circle, draw = black, fill = black, inner sep = 0.5mm, label=left:{3}] {};
	\node [circle, draw = black, fill = black, inner sep = 0.5mm, label=above:{4}] (3) at (-1,2) {};
	\node [circle, draw = black, fill = black, inner sep = 0.5mm, label=above:{6}] (4) at (1,2) {};
    \draw (1)--(2);
    \draw (2)--(3);
    \draw (2)--(4);
\node (v1) at (0.5,0) [circle, draw = black, fill = black, inner sep = 0.5mm, label=below:{2}] {};
\draw (v1) -- (2);
\node [circle, draw = black, fill = black, inner sep = 0.5mm, label=above:{5}] (v2) at (0,2) {};
\draw (2) -- (v2);
\end{tikzpicture}\quad\quad\quad\quad \begin{tikzpicture}
    \node at (0,0) {$\begin{bmatrix}
    0 & 0 & \ast & \ast & \ast & \ast \\
    0 & 0 & \ast & \ast & \ast & \ast \\
    0 & 0 & 0 & \ast & \ast & \ast \\
    0 & 0 & 0 & 0 & 0 & 0 \\
    0 & 0 & 0 & 0 & 0 & 0 \\
    0 & 0 & 0 & 0 & 0 & 0 
\end{bmatrix}$};
\end{tikzpicture}$$
\caption{$\mathcal{P}=[6]$ with $1,2\prec3\prec4,5,6$ (left) and the nilpotent Lie poset algebra $\mathfrak{g}^{\prec}(\mathcal{P})$ (right)}\label{fig:poset}
\end{figure}
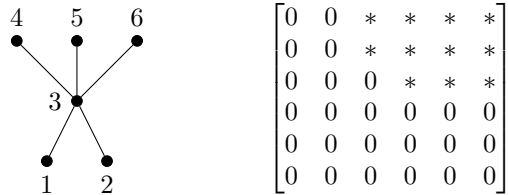
\end{ex}

Given a finite poset $\mathcal{P}$, the \demph{nilpotent Lie poset algebra} $\mathfrak{g}^{\prec}(\mathcal{P})$ is the subalgebra of $\mathfrak{gl}(|\mathcal{P}|,\mathbb{C})$ whose underlying vector space is spanned by the matrices $E_{p,q}$ for $p,q\in\mathcal{P}$ satisfying $p\prec q$. For an example, the nilpotent Lie poset algebra $\mathfrak{g}^{\prec}(\mathcal{P})$ for $\mathcal{P}$ the poset of Example~\ref{ex:posnot} consists of all matrices in $\mathfrak{gl}(6,\mathbb{C})$ of the form illustrated in Figure~\ref{fig:poset} (right) where $\ast$'s denote potential nonzero entries.

\begin{remark}
    The definition of nilpotent Lie poset given above defers slightly from that given in the papers \cite{breadthposet,indexnilposet,tBCD} where it is assumed that for all finite posets $(\mathcal{P},\preceq)$ the ordering $\preceq$ respects the natural ordering $\le$; that is, it is assumed that $a\preceq b$ in $\mathcal{P}$ implies that $a\le b$. Using a linear extension, it follows that every poset considered here is isomorphic to one whose ordering respects $\le$. Consequently, since isomorphic posets have isomorphic nilpotent Lie poset algebras, it follows that the results of \cite{indexnilposet} apply to the nilpotent Lie poset algebras considered here. 
\end{remark}

Moving forward, we find that certain posets and their corresponding nilpotent Lie poset algebras play an important role. For a composition $(a_1,a_2,\hdots,a_\ell)$ of $N$ with $S(\mathbf{a})=A_1\cup\cdots\cup A_n$, define $\mathcal{P}(a_1,a_2,\hdots,a_\ell)$ to be the poset on $[N]$ whose covering relations are given by $p\precdot q$ whenever $p\in A_j$ and $q\in A_{j+1}$ for $j\in [\ell-1]$. We denote $\mathfrak{g}^{\prec}(\mathcal{P}(a_1,a_2,\hdots,a_\ell))$ by $\mathfrak{g}^{\prec}(a_1,a_2,\hdots,a_\ell)$.

Now, as we will show in the following section, the nilradical of a (type-A) seaweed algebra $\mathfrak{s}$ is the direct sum of $\mathrm{Z}(\mathfrak{s})$ with a nilpotent Lie poset algebra. Thus, a poset-theoretic index formula for nilpotent Lie poset algebras found in \cite{indexnilposet} will play a crucial role in our work here. To state the formula, we require the following notation. If $\mathcal{P}$ is a finite poset and $p\in\mathcal{P}$, then $\mathrm{D}(\mathcal{P},p):=|\{q\in \mathcal{P}~|~q\prec p\}|$ and $\mathrm{U}(\mathcal{P},p):=|\{q\in \mathcal{P}~|~p\prec q\}|.$

\begin{theorem}[Theorem~4, \cite{indexnilposet}]\label{thm:indexnil}
    If $\mathcal{P}$ is a finite poset, then
$$\ind\mathfrak{g}^{\prec}(\mathcal{P})=|\mathrm{Rel}(\mathcal{P})|-2\sum_{p\in\mathcal{P}\backslash Ext(\mathcal{P})}\min(\mathrm{D}(\mathcal{P},p),\mathrm{U}(\mathcal{P},p)).$$ 
\end{theorem}

\section{Nilradicals to nilpotent Lie poset algebras}\label{sec:posetalg}

In this section, we show that the nilradical of a (type-A) seaweed algebra $\mathfrak{s}$ is isomorphic to the direct sum of $\mathrm{Z}(\mathfrak{s})$ with a nilpotent Lie poset algebra. Moreover, we identify a choice of corresponding poset. To start, we establish the following result which characterizes $\mathrm{Z}(\mathfrak{s})$.

\begin{theorem}\label{thm:center}
    Let $\mathfrak{s}$ be a (type-A) seaweed algebra and $S\subseteq\mathfrak{s}$ consist of all elements of the form $\sum c_iE_{i,i}$ where $c_j=c_k$ when $v_j$ and $v_k$ belong to the same central component of $\mathrm{M}(\mathfrak{s})$. Then $\mathrm{Z}(\mathfrak{s})=S$.
\end{theorem}
\begin{proof}
    Take $x=\sum c_{i,j}E_{i,j}\in\mathrm{Z}(\mathfrak{s})$. Since $$[E_{i,i}-E_{j,j},x]=\sum_{i\neq k} c_{i,k}E_{i,k}+\sum_{j\neq k} c_{k,j}E_{k,j}-\sum_{i\neq k} c_{k,i}E_{k,i}-\sum_{j\neq k} c_{j,k}E_{j,k}=0,$$ it follows that $x=\sum c_iE_{i,i}$. Next, take $E_{i,j}\in\mathfrak{s}$ so that $v_i$ and $v_j$ belong to the same central component of $\mathrm{M}(\mathfrak{s})$ by Lemma~\ref{lem:cc} (a). Then $[x,E_{i,j}]=(c_i-c_j)E_{i,j}=0$, i.e., $c_i=c_j$. Thus, it follows that $x\in S$ and $\mathrm{Z}(\mathfrak{s})\subseteq S$. Now, take $x=\sum c_iE_{i,i}\in S$. Evidently, $[x,E_{i,i}]=0$. If $E_{i,j}\in\mathfrak{s}$, then by Lemma~\ref{lem:cc} (a) we have that $v_i$ and $v_j$ belong to the same central component of $\mathrm{M}(\mathfrak{s})$. Consequently, $[x,E_{i,j}]=(c_i-c_j)E_{i,j}=0$. Therefore, $x\in \mathrm{Z}(\mathfrak{s})$ and $S\subseteq\mathrm{Z}(\mathfrak{s})$. The result follows.
\end{proof}

Next, the following theorem characterizes the nilradical of a (type-A) seaweed algebra.

\begin{theorem}\label{thm:nilrad}
    If $\mathfrak{s}$ is a (type-A) seaweed algebra, then $$\mathfrak{n}(\mathfrak{s})=\mathrm{span}\{E_{p,q}\in\mathfrak{s}~|~E_{q,p}\notin\mathfrak{s}\}\oplus\mathrm{Z}(\mathfrak{s}).$$
\end{theorem}
\begin{proof}
    Assume that either $\mathfrak{s}=\mathfrak{p}\frac{a_1|\cdots|a_n}{b_1|\cdots|b_m}$ or $\mathfrak{p}^A\frac{a_1|\cdots|a_n}{b_1|\cdots|b_m}$ with $N=\sum_{i=1}^na_i=\sum_{i=1}^mb_i$ and let $$\mathfrak{k}=\mathrm{span}\{E_{p,q}\in\mathfrak{s}~|~E_{q,p}\notin\mathfrak{s}\}\oplus\mathrm{Z}(\mathfrak{s}).$$ To begin, we will establish that $\mathfrak{n}(\mathfrak{s})$ is contained in $\mathfrak{k}$. Towards this end, if $E_{p,q},E_{q,p}\in\mathfrak{s}$ and $x=\sum c_{i,j}E_{i,j}\in \mathfrak{n}(\mathfrak{s})$, then we claim that $c_{p,q}=c_{q,p}=0$. To see this, note that 
    \begin{multline*}
        \left[E_{q,q}-\frac{1}{N-1}\sum_{i\neq q}E_{i,i},\left[E_{p,p}-\frac{1}{N-1}\sum_{i\neq p}E_{i,i},x\right]\right]\\
        =-c_{p,q}\left(1+\frac{1}{N-1}\right)^2E_{p,q}-c_{q,p}\left(1+\frac{1}{N-1}\right)^2E_{q,p}\in \mathfrak{n}(\mathfrak{s})
    \end{multline*} so that $c_{p,q}E_{p,q}+c_{q,p}E_{q,p}\in \mathfrak{n}(\mathfrak{s})$. Assume for the sake of contradiction that $c_{q,p}\neq 0$, the argument for $c_{p,q}\neq 0$ being similar. Then, $$[E_{p,q},c_{p,q}E_{p,q}+c_{q,p}E_{q,p}]=c_{q,p}(E_{p,p}-E_{q,q})\in \mathfrak{n}(\mathfrak{s})$$ so that $E_{p,p}-E_{q,q}\in \mathfrak{n}(\mathfrak{s})$. Thus, since $E_{p,p}-E_{q,q},c_{p,q}E_{p,q}+c_{q,p}E_{q,p}\in \mathfrak{n}(\mathfrak{s})$ with $c_{q,p}\neq 0$ and $$\mathrm{ad}^n_{E_{p,p}-E_{q,q}}(c_{p,q}E_{p,q}+c_{q,p}E_{q,p})=2^n(c_{p,q}E_{p,q}+(-1)^nc_{q,p}E_{q,p})\neq 0,$$ it follows that $\mathrm{ad}_{E_{p,p}-E_{q,q}}$ is not a nilpotent endomorphism of $\mathfrak{n}(\mathfrak{s})$ so that, applying Engel's Theorem, $\mathfrak{n}(\mathfrak{s})$ is not nilpotent, which is a contradiction. Consequently, if $E_{p,q},E_{q,p}\in\mathfrak{s}$ and $x=\sum c_{i,j}E_{i,j}\in \mathfrak{n}(\mathfrak{s})$, then $c_{p,q}=c_{q,p}=0$, as claimed. Now, in light of Theorem~\ref{thm:center}, to establish that $\mathfrak{n}(\mathfrak{s})$ is contained in $\mathfrak{k}$, it remains to show that if $x=\sum c_{i,j}E_{i,j}\in\mathfrak{n}(\mathfrak{s})$, then $c_{p,p}=c_{q,q}$ for all $p,q\in [N]$ such that the vertices $v_p$ and $v_q$ belong to the same central component of $\mathrm{M}(\mathfrak{s})$. Assume for a contradiction that $x=\sum c_{i,j}E_{i,j}\in\mathfrak{n}(\mathfrak{s})$ with $c_{p,p}\neq c_{q,q}$ for some $p<q\in [N]$ with $v_p$ and $v_q$ belonging to the same central component of $\mathrm{M}(\mathfrak{s})$. Note that we may assume that $q=p+1$ as, under our assumption, such a pair must exist. Moreover, note that if $E_{r,t}\in \mathfrak{s}$, $E_{t,r}\notin \mathfrak{s}$, and $c_{r,t}\neq 0$, then $$\left[E_{t,t}-\frac{1}{N-1}\sum_{i\neq t}E_{i,i},\left[E_{r,r}-\frac{1}{N-1}\sum_{i\neq r}E_{i,i},x\right]\right]=-c_{r,t}\left(1+\frac{1}{N-1}\right)^2E_{r,t}\in \mathfrak{n}(\mathfrak{s})$$ so that $E_{r,t}\in \mathfrak{n}(\mathfrak{s})$ and $0\neq x_d=\sum c_{i,i}E_{i,i}=x-\sum_{i\neq j}c_{i,j}E_{i,j}\in\mathfrak{n}(\mathfrak{s})$. Applying Lemma~\ref{lem:cc} (b), since $v_p$ and $v_{p+1}$ belong to the same central component of $\mathrm{M}(\mathfrak{s})$, we have that either $E_{p,p+1}\in\mathfrak{s}$ or $E_{p+1,p}\in\mathfrak{s}$. Without loss of generality, assume that $y=E_{p,p+1}\in\mathfrak{s}$. Then we have that $[x_d,y]=(c_{p,p}-c_{p+1,p+1})y\in\mathfrak{n}(\mathfrak{s})$, i.e., $y\in \mathfrak{n}(\mathfrak{s})$, and $$\mathrm{ad}_{x_d}^n(y)=(c_{p,p}-c_{p+1,p+1})^ny\in \mathfrak{n}(\mathfrak{s})$$ for all $n\in\mathbb{Z}_{>0}$ with $c_{p,p}-c_{p+1,p+1}\neq 0$; but then, arguing as above using Engel's Theorem, $\mathfrak{n}(\mathfrak{s})$ is not nilpotent, which is a contradiction. Therefore, we have shown that $\mathfrak{n}(\mathfrak{s})$ is contained in $\mathfrak{k}$.

    Having established that $\mathfrak{n}(\mathfrak{s})$ is contained in $\mathfrak{k}$, to finish the proof, it suffices to show that $$\mathfrak{n}=\mathrm{span}\{E_{p,q}\in\mathfrak{s}~|~E_{q,p}\notin\mathfrak{s}\}$$ forms a nilpotent ideal of $\mathfrak{s}$. To see that $\mathfrak{n}$ forms an ideal of $\mathfrak{s}$, take $E_{x,y},E_{y,z}\in\mathfrak{s}$ with $E_{x,y}\in\mathfrak{n}$ or $E_{y,z}\in\mathfrak{n}$. If $[E_{x,y},E_{y,z}]=E_{x,z}\notin\mathfrak{n}$, then $E_{z,x}\in\mathfrak{s}$ so that $[E_{z,x},E_{x,y}]=E_{z,y}\in\mathfrak{s}$ and $[E_{y,z},E_{z,x}]=E_{y,x}\in\mathfrak{s}$, contradicting our assumption that $E_{x,y}\in\mathfrak{n}$ or $E_{y,z}\in\mathfrak{n}$. Thus, $E_{x,z}\in\mathfrak{n}$ and it follows that $\mathfrak{n}$ is an ideal of $\mathfrak{s}$. Now, to see that $\mathfrak{n}$ forms a nilpotent ideal, we show that it is a nilpotent Lie poset algebra. Let $\mathcal{P}_\mathfrak{s}=[N]$ and $\prec_\mathfrak{s}$ be the relation on $\mathcal{P}_\mathfrak{s}=[N]$ defined by $x\prec_\mathfrak{s}y$ for $x,y\in [N]$ if and only if $E_{x,y}\in\mathfrak{n}$. Since $\mathfrak{n}$ is an ideal of $\mathfrak{s}$, we have that $E_{x,y},E_{y,z}\in\mathfrak{n}$ implies $[E_{x,y},E_{y,z}]=E_{x,z}\in\mathfrak{n}$. Consequently, the relation $\prec_\mathfrak{s}$ is transitive. Moreover, since $E_{x,y}\in\mathfrak{n}$ implies that $E_{y,x}\notin\mathfrak{n}$, it follows that $\prec_\mathfrak{s}$ is antisymmetric. Therefore, $(\mathcal{P}_\mathfrak{s},\preceq_\mathfrak{s})$ is a poset and $\mathfrak{n}=\mathfrak{g}^{\prec}(\mathcal{P}_\mathfrak{s})$. The result follows.
\end{proof}

\begin{ex}
    For $\mathfrak{s}=\mathfrak{p}\frac{2|3|1|2|2}{7|3}$ and $\mathfrak{p}^A\frac{2|3|1|2|2}{7|3}$, in Figure~\ref{fig:nilrad} (left) we illustrate the elements of $\mathfrak{s}$ where $\ast$'s denote potential nonzero entries. Considering Theorem's~\ref{thm:center} and~\ref{thm:nilrad}, the nilradical $\mathfrak{n}(\mathfrak{s})=\mathrm{Z}(\mathfrak{s})\oplus\mathfrak{g}^\prec(\mathfrak{s})$ is contained within the span of the elements along the diagonal together with those corresponding to $\ast$'s lying outside of the boxed regions. More specifically, $\mathrm{Z}(\mathfrak{s})$ is spanned by the diagonal elements which have fixed coefficient within each boxed region, while $\mathfrak{g}^\prec(\mathfrak{s})$ is spanned by the elements corresponding to $\ast$'s lying outside of the boxed regions. In Figure~\ref{fig:nilrad} (right) we illustrate $\mathcal{P}_\mathfrak{s}$.
    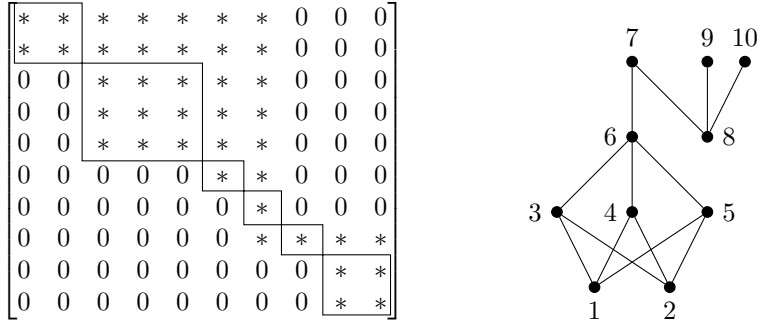
\begin{figure}[h]
        $$\begin{tikzpicture}
            \def\Node{\node [circle, draw = black, inner sep=2pt]}
     \node (1) at (0,0) {$\begin{bmatrix}
         \ast & \ast & \ast & \ast & \ast & \ast & \ast & 0 & 0 & 0 \\
         \ast & \ast & \ast & \ast & \ast & \ast & \ast & 0 & 0 & 0 \\
         0 & 0 & \ast & \ast & \ast & \ast & \ast & 0 & 0 & 0 \\
         0 & 0 & \ast & \ast & \ast & \ast & \ast & 0 & 0 & 0 \\
         0 & 0 & \ast & \ast & \ast & \ast & \ast & 0 & 0 & 0 \\
         0 & 0 & 0 & 0 & 0 & \ast & \ast & 0 & 0 & 0 \\
         0 & 0 & 0 & 0 & 0 & 0 & \ast & 0 & 0 & 0 \\
         0 & 0 & 0 & 0 & 0 & 0 & \ast & \ast & \ast & \ast \\
         0 & 0 & 0 & 0 & 0 & 0 & 0 & 0 & \ast & \ast \\
         0 & 0 & 0 & 0 & 0 & 0 & 0 & 0 & \ast & \ast
     \end{bmatrix}$};
     \draw (0,0)--(0,1.3)--(-2.5,1.3)--(-2.5,2.1)--(-1.6,2.1)--(-1.6,0)--(0.55,0)--(0.55,-0.4)--(1.05,-0.4)--(1.05,-0.85)--(1.6,-0.85)--(1.6,-1.25)--(2.5,-1.25)--(2.5,-2.05)--(1.6,-2.05)--(1.6,-1.25)--(1.05,-1.25)--(1.05,-0.85)--(0.55,-0.85)--(0.55,-0.4)--(0,-0.4)--(0,0);
    \end{tikzpicture}\quad\quad\quad\quad \begin{tikzpicture}
	\node (1) at (0, 0) [circle, draw = black, fill=black, inner sep = 0.5mm, label=below:{$1$}] {};
 \node (2) at (1, 0) [circle, draw = black, fill=black, inner sep = 0.5mm, label=below:{$2$}] {};
 \node (3) at (-0.5, 1) [circle, draw = black, fill=black, inner sep = 0.5mm, label=left:{$3$}] {};
 \node (4) at (0.5, 1) [circle, draw = black, fill=black, inner sep = 0.5mm, label=left:{$4$}] {};
 \node (5) at (1.5, 1) [circle, draw = black, fill=black, inner sep = 0.5mm, label=right:{$5$}] {};
 \node (6) at (0.5, 2) [circle, draw = black, fill=black, inner sep = 0.5mm, label=left:{$6$}] {};
 \node (7) at (0.5,3) [circle, draw = black, fill=black, inner sep = 0.5mm, label=above:{$7$}] {};
 \node (8) at (1.5, 2) [circle, draw = black, fill=black, inner sep = 0.5mm, label=right:{$8$}] {};
 \node (9) at (1.5, 3) [circle, draw = black, fill=black, inner sep = 0.5mm, label=above:{$9$}] {};
 \node (10) at (2, 3) [circle, draw = black, fill=black, inner sep = 0.5mm, label=above:{$10$}] {};
 \draw (1)--(3)--(6)--(4)--(2)--(3);
 \draw (1)--(5)--(6);
 \draw (1)--(4);
 \draw (2)--(5);
 \draw (6)--(7)--(8)--(9);
 \draw (8)--(10);
\end{tikzpicture}$$
        \caption{$\mathfrak{n}(\mathfrak{s})$ and $\mathcal{P}_\mathfrak{s}$ for $\mathfrak{s}=\mathfrak{p}\frac{2|3|1|2|2}{7|3}$ and $\mathfrak{p}^A\frac{2|3|1|2|2}{7|3}$}\label{fig:nilrad}
    \end{figure}
\end{ex}

Given a (type-A) seaweed algebra $\mathfrak{s}$, we now show how the poset $\mathcal{P}_\mathfrak{s}$ defined in the proof of Theorem~\ref{thm:nilrad} for which $\mathfrak{n}(\mathfrak{s})=\mathfrak{g}^{\prec}(\mathcal{P}_{\mathfrak{s}})\oplus\mathrm{Z}(\mathfrak{s})$ can be determined from an associated ``meander". Let $\mathfrak{s}=\mathfrak{p}\frac{a_1|\cdots|a_n}{b_1|\cdots|b_m}$ or $\mathfrak{p}^A\frac{a_1|\cdots|a_n}{b_1|\cdots|b_m}$, $$N=\sum_{i=1}^na_i=\sum_{i=1}^mb_i,\quad \mathbf{a}=(a_1,\hdots,a_n),\quad \mathbf{b}=(b_1,\hdots,b_m),\quad S(\mathbf{a})=A_1\cup A_2\cup\cdots\cup A_n,$$ and $$S(\mathbf{b})=B_1\cup B_2\cup\cdots\cup B_m.$$ Moreover, define $$C(\mathfrak{s}):=C_1\cup\cdots\cup C_\ell$$ to be the ordered set partition formed from the sets $A_i\cap B_j$ for $i\in [n]$ and $j\in [m]$, where $\max C_k=1+\min C_{k+1}$ for $k\in[\ell-1]$. Considering our definitions of $\mathfrak{p}\frac{\mathbf{a}}{\mathbf{b}}$ and $\mathfrak{p}^A\frac{\mathbf{a}}{\mathbf{b}}$ given in Section~\ref{sec:prelimsw}, it follows that $E_{p,q}\in \mathfrak{s}$ for $p>q$ if and only if there exists $i\in [n]$ for which $p,q\in A_i$. Similarly, $E_{p,q}\in \mathfrak{s}$ for $p<q$ if and only if there exists $i\in [m]$ for which $p,q\in B_i$. Consequently, considering Theorem~\ref{thm:nilrad}, for $E_{p,q}\in\mathfrak{s}$, we have $E_{p,q}\notin \mathfrak{n}(\mathfrak{s})$ if and only if there exists $i\in [\ell]$ for which $p,q\in C_i$.

Now, for the elements of $\mathfrak{n}(\mathfrak{s})$, take $p\in [N]$. Assume that $p\in C_{i_p}$. Moreover, assume that $j_p\in [n]$ and $k_p\in [m]$ are the unique elements for which $C_{i_p}\subseteq A_{j_p},B_{k_p}$. Then we have that $E_{p,q}\in\mathfrak{n}(\mathfrak{s})$ if and only if $q\in B_{k_p}\backslash C_{i_p}$ with $p<q$ or $q\in A_{j_p}\backslash C_{i_p}$ with $p>q$. Similarly, $E_{q,p}\in\mathfrak{n}(\mathfrak{s})$ if and only if $q\in A_{j_p}\backslash C_{i_p}$ with $p<q$ or $q\in B_{k_p}\backslash C_{i_p}$ with $p>q$. We can organize this information nicely using a modified version of the meander construction discussed in Section~\ref{sec:prelim}. For the modified meander, denoted $\mathrm{M}_\mathfrak{p}(\mathfrak{s})$, we proceed as follows.
\begin{itemize}
    \item[$(1)$] Write the numbers 1 through $N$ in increasing order from left to right along a horizontal line $L$.
    \item[$(2)$] For $1\le j\le n$ draw a vertical line between the numbers $\sum_{i=1}^ja_i$ and $1+\sum_{i=1}^ja_i$ which lies above $L$.
    \item[$(3)$] Similarly, for $1\le j\le m$ draw a vertical line between the numbers $\sum_{i=1}^jb_i$ and $1+\sum_{i=1}^jb_i$ which lies below $L$.
    \item[$(4)$] Circle each collection of numbers which lie between consecutive vertical lines (such lines do not have to both lie above or below $L$). Note that the collections of circled numbers correspond to the $C_i$ for $i\in [\ell]$. Consequently, we will refer to the collection of circled numbers corresponding to $C_i$ in $\mathrm{M}_\mathfrak{p}(\mathfrak{s})$ as $C_i$ for $i\in [\ell]$. Context will make clear whether we are referring to the set or the circled collection of numbers in $\mathrm{M}_\mathfrak{p}(\mathfrak{s})$.
    \item[$(5)$] Finally, for $i\in[\ell-1]$, if the collections of circled numbers $C_i$ and $C_{i+1}$ are separated by a single vertical line which lies above $L$, then draw a directed edge below from $C_i$ to $C_{i+1}$; if the collections are separated by a single vertical line which lies below $L$, then draw a directed edge above from $C_{i+1}$ to $C_{i}$; and finally, if the collections are separated by two vertical lines, one lying above $L$ and the other below, then draw no edges between the collections.
\end{itemize}
Writing $p\rightarrow_{\mathfrak{s}} q$ if there is a directed path from the circled collection containing $p$ in $\mathrm{M}_\mathfrak{p}(\mathfrak{s})$ to that containing $q$, considering our discussion above, it is straightforward to verify that $E_{p,q}\in\mathfrak{n}(\mathfrak{s})$ if and only if $p\rightarrow_{\mathfrak{s}} q$. Thus, $\mathcal{P}_\mathfrak{s}$ is the poset whose partial order is given by $\prec_\mathfrak{s}=\rightarrow_{\mathfrak{s}}$. Having discussed the meander construction of $\mathcal{P}_\mathfrak{s}$, from the proof of Theorem~\ref{thm:nilrad} we now state the following.

\begin{theorem}\label{thm:nilposet}
    If $\mathfrak{s}$ is a (type-A) seaweed algebra, then $\mathfrak{n}(\mathfrak{s})\cong \mathfrak{g}^\prec(\mathcal{P}_\mathfrak{s})\oplus\mathrm{Z}(\mathfrak{s})$.
\end{theorem}

Considering the definition of $\mathcal{P}_\mathfrak{s}$ in terms of $\mathrm{M}_\mathfrak{p}(\mathfrak{s})$ given above, the following is immediate.

\begin{corollary}\label{cor:parbposet}
    If $\mathfrak{s}=\mathfrak{p}\frac{a_1|\cdots|a_m}{N}$ or $\mathfrak{p}\frac{N}{a_m|\cdots|a_1}$, then $\mathfrak{n}(\mathfrak{s})\cong \mathfrak{g}^\prec(\mathcal{P}_\mathfrak{s})\oplus\mathrm{Z}(\mathfrak{s})$ with $\mathfrak{g}^\prec(\mathcal{P}_\mathfrak{s})\cong\mathfrak{g}^{\prec}(a_1,\hdots,a_m)$.
\end{corollary}

\begin{ex}
    For $\mathfrak{s}=\mathfrak{p}\frac{2|3|1|2|2}{7|3}$ and $\mathfrak{p}^A\frac{2|3|1|2|2}{7|3}$, the meander $\mathrm{M}_\mathfrak{p}(\mathfrak{s})$ is illustrated in Figure~\ref{fig:MsandPs} and the Hasse diagram of $\mathcal{P}_\mathfrak{s}$ is illustrated in Figure~\ref{fig:nilrad} (right).
    \begin{figure}[h]
        $$\begin{tikzpicture}
            \def\Node{\node [circle, draw = black, inner sep=2pt]}
     \Node (1) at (0,0) {$1\quad 2$};
	\Node (2) at (1.5,0) {$3\quad 4\quad 5$};
	\Node (3) at (2.75,0) {$6$};
	\Node (4) at (3.5,0) {$7$};
	\Node (5) at (4.25,0) {$8$};
	\Node (6) at (5.25,0) {$9\quad 10$};
    \draw (0.65,0)--(0.65,0.5);
    \draw (2.35,0)--(2.35,0.5);
    \draw (3.15,0)--(3.15,0.5);
    \draw (3.9,0)--(3.9,-0.5);
    \draw (4.6,0)--(4.6,0.5);
	 \draw[->] (1) to[bend right=50] (2);
	\draw[->] (2) to[bend right=50] (3);
	\draw[->] (3) to[bend right=50] (4);
	 \draw[<-] (4) to[bend left=50] (5);
	\draw[->] (5) to[bend right=50] (6);
        \end{tikzpicture}$$
        \caption{$\mathrm{M}_{\mathfrak{p}}(\mathfrak{s})$ for $\mathfrak{s}=\mathfrak{p}\frac{2|3|1|2|2}{7|3}$ and $\mathfrak{p}^A\frac{2|3|1|2|2}{7|3}$}\label{fig:MsandPs}
    \end{figure}
\end{ex}

Now, to describe the form of the posets $\mathcal{P}_\mathfrak{s}$, note that considering our definition of $\mathcal{P}_\mathfrak{s}$ given in terms of $\mathrm{M}_\mathfrak{p}(\mathfrak{s})$ above, we have that $p\precdot q$ in $\mathcal{P}_\mathfrak{s}$ if and only if there exists $i,j\in [\ell]$ such that $|i-j|=1$, $p\in C_i$, $q\in C_j$, and there is an edge directed from $C_i$ to $C_j$. Thus, reading off covering relations from left to right in $\mathrm{M}_\mathfrak{p}(\mathfrak{s})$, for $i\in [\ell-1]$, we have that either
\begin{itemize}
    \item[$(a)$] $C_i$ is connected by $C_{i+1}$ by edge directed towards $C_{i+1}$ so that the elements of $C_i$ are covered by the elements of $C_{i+1}$;
    \item[$(b)$] $C_i$ is connected by $C_{i+1}$ by edge directed towards $C_{i}$ so that the elements of $C_i$ cover the elements of $C_{i+1}$; or
    \item[$(c)$] $C_i$ is not connected to $C_{i+1}$ by an edge so that the elements of $C_j$ for $j\in [i]$ are not related to the elements of $C_k$ for $j+1\le k\le \ell$.
\end{itemize}
Consequently, defining the \demph{connected components} of $M_\mathfrak{p}(\mathfrak{s})$ to be the collections of $C_j$ for which each pair $C_i$ and $C_{i+1}$ are connected by an edge, the connected components of $M_\mathfrak{p}(\mathfrak{s})$ correspond to the connected components of $\mathcal{P}_\mathfrak{s}$. Moreover, identifying $i$ with $v_i$ for $i\in\mathcal{P}_\mathfrak{s}$, the connected components of $\mathrm{M}_\mathfrak{p}(\mathfrak{s})$ coincide with the central components of $\mathrm{M}(\mathfrak{s})$. To describe the form of the connected components of $\mathcal{P}_\mathfrak{s}$, assume that it is connected. Then $\mathrm{M}_\mathfrak{p}(\mathfrak{s})$ has a single connected component and there exists $1=i_0<i_1<\cdots<i_d=\ell$ such that all edges adjacent to $C_{i_j}$ are either directed into $C_{i_j}$ (i.e., $C_{i_j}$ is a \demph{sink}) or out of $C_{i_j}$ (i.e., $C_{i_j}$ is a \demph{source}) for $j\in [d]$. Note that if $C_{i_0}$ is a source (resp., sink), then $C_{i_j}$ is a source (resp., sink) for $j\in [d]$ even, while $C_{i_j}$ is a sink (resp., source) for $j\in [d]$ odd. In addition, for $j\in [d-1]$, if $C_{j}$ is a source (resp., sink), then all edges between $C_{i}$ and $C_{i+1}$ for $i_j\le i\le i_{j+1}$ are directed towards decreasing (resp., increasing) indices. Now, let $\mathcal{P}_j$ denote the poset isomorphic to the subposet of $\mathcal{P}_\mathfrak{s}$ induced by the elements $\bigcup_{i={i_j}}^{i_{j+1}}C_i$ for $j\in [d-1]$. Then in the case $C_{i_0}$ is a source (resp., sink), we have that $$\mathcal{P}_j\cong\mathcal{P}(|C_{i_{j+1}}|,|C_{i_{j+1}-1}|\hdots,|C_{i_j}|)$$ for $j\in [d-1]$ even (resp., odd), $$\mathcal{P}_j\cong\mathcal{P}(|C_{i_{j}}|,|C_{i_{j}+1}|\hdots,|C_{i_{j+1}}|)$$ for $j\in [d-1]$ odd (resp., even). Note that under the assumption that $C_{i_0}$ is a source (resp., sink) $\mathcal{P}_\mathfrak{s}$ is built from the $\mathcal{P}_j$ by identifying maximal (resp., minimal) elements of $\mathcal{P}_j$ and $\mathcal{P}_{j+1}$ for $j\in [d-2]$ even, and identifying minimal (resp., maximal) elements of $\mathcal{P}_j$ and $\mathcal{P}_{j+1}$ for $j\in [d-2]$ odd. When $C_{i_0}$ is a source we denote the poset built from the $\mathcal{P}_j$ for $j\in [d-1]$ described above by $\mathrm{out}(\mathcal{P}_1,\hdots,\mathcal{P}_d)$, while when $C_{i_0}$ is a sink we denote the poset by $\mathrm{in}(\mathcal{P}_1,\hdots,\mathcal{P}_d)$.

\begin{ex}
    Let $\mathfrak{s}=\mathfrak{p}\frac{2|3|1|2|2}{7|3}$ or $\mathfrak{p}^A\frac{2|3|1|2|2}{7|3}$. Denote the subposet of $\mathcal{P}_\mathfrak{s}$ induced by the elements $\{1,2,3,4,5,6,7\}$ by $\mathcal{P}_1$, that induced by the elements $\{7,8\}$ by $\mathcal{P}_2$, and that induced by the elements $\{7,8,9,10\}$ by $\mathcal{P}_3$. Then $$\mathcal{P}_1\cong\mathcal{P}(2,3,1,1),\quad \mathcal{P}_2\cong\mathcal{P}(1,1),\quad \mathcal{P}_3\cong\mathcal{P}(1,2),\quad\text{and}\quad \mathcal{P}_\mathfrak{s}=\mathrm{out}(\mathcal{P}_1,\mathcal{P}_2,\mathcal{P}_3).$$
\end{ex}

\begin{remark}
    Note that if $\mathcal{P}_\mathfrak{s}=\mathrm{in}(\mathcal{P}_1,\hdots,\mathcal{P}_d)$, then we assume that $1\in\mathcal{P}_1$ and the first edge in $\mathrm{M}_\mathfrak{p}(\mathfrak{s})$ from left to right is above and directed towards $C_1$. Moreover, if $\mathcal{P}_\mathfrak{s}=\mathrm{in}(\mathcal{P}_1)$ with $\mathcal{P}_1\cong\mathcal{P}(a_1,\hdots,a_m)$, then $\mathfrak{s}=\mathfrak{p}\frac{N}{a_m|\cdots|a_1}$ or $\mathfrak{p}^A\frac{N}{a_m|\cdots|a_1}$. Similarly, if $\mathcal{P}_\mathfrak{s}=\mathrm{out}(\mathcal{P}_1,\hdots,\mathcal{P}_d)$, then we assume that $1\in\mathcal{P}_1$ and the first edge in $\mathrm{M}_\mathfrak{p}(\mathfrak{s})$ from left to right is below and directed away from $C_1$. Moreover, if $\mathcal{P}_\mathfrak{s}=\mathrm{out}(\mathcal{P}_1)$ with $\mathcal{P}_1\cong\mathcal{P}(a_1,\hdots,a_m)$, then $\mathfrak{s}=\mathfrak{p}\frac{a_1|\cdots|a_m}{N}$ or $\mathfrak{p}^A\frac{a_1|\cdots|a_m}{N}$.
\end{remark}

So, intuitively, we have shown that the connected components of the posets $\mathcal{P}_\mathfrak{s}$ are isomorphic to posets which can be built from those of the form $\mathcal{P}(a_1,\hdots,a_m)$ by alternatingly identifying minimal and maximal elements of each. While we won't prove it, it is not hard to show that such posets are always isomorphic to $\mathcal{P}_\mathfrak{s}$ for some seaweed algebra $\mathfrak{s}$.

\section{Index}\label{sec:index}

In this section, we establish a combinatorial formula for the index of nilradicals of (type-A) seaweed algebras. Our new formula depends on an edge-weighted meander graph, denoted $\mathrm{M}_\mathfrak{n}(\mathfrak{s})$, which we define as follows. 

Letting $\mathfrak{s}=\mathfrak{p}\frac{\mathbf{a}}{\mathbf{b}}$ or $\mathfrak{p}^A\frac{\mathbf{a}}{\mathbf{b}}$ draw $\mathrm{M}(\mathfrak{s})$. Denoting by $L$ the horizontal line along which the vertices of $\mathrm{M}(\mathfrak{s})$ are drawn, add vertical lines above (resp., below) $L$ which partition the vertices of $\mathrm{M}_\mathfrak{n}(\mathfrak{s})$ into their top (resp., bottom) blocks. Then for each top block of $\mathrm{M}(\mathfrak{s})$ from left to right, followed by the bottom blocks, also from left to right, proceed as follows. Denoting the current top (resp., bottom) block under consideration by $B$, assume that the collection of vertices of $B$ are partitioned into the bottom (resp., top) blocks $V_1,\hdots,V_\ell$.
    \begin{enumerate}
        \item[$(a)$] If $\ell=1$, then assign all top (resp., bottom) edges connecting vertices of $V_1$ the weight of 0.
        \item[$(b)$] If $\ell>1$, $M=\max\{|V_1|,|V_\ell|\}$, and $m=\min\{|V_1|,|V_\ell|\}$, then weight the top (resp., bottom) edges connecting the $m$ first vertices of $V_1$ with the $m$ last vertices of $V_\ell$ by $M$.
    \end{enumerate}
Now, repeat $(a)$ or $(b)$ for the collection $V_1',\hdots,V'_{\ell}$ where $V_i'$ consists of the vertices of $V_i$ not adjacent to a weighted top (resp., bottom) edge and we omit any $V_i'=\emptyset$. If $V'_i=\emptyset$ for all $i\in [\ell]$, then the edges corresponding to the top (resp., bottom) block $B$ are weighted and we move to the next top or bottom block. Continue until all edges have been assigned a weight.

\begin{ex}\label{ex:Mn1}
    For $\mathfrak{s}_1=\mathfrak{p}\frac{2|3|1|2|2}{7|3}$ or $\mathfrak{p}^A\frac{2|3|1|2|2}{7|3}$, the meander $\mathrm{M}_\mathfrak{n}(\mathfrak{s}_1)$ is illustrated in Figure~\ref{fig:Mns1}.
    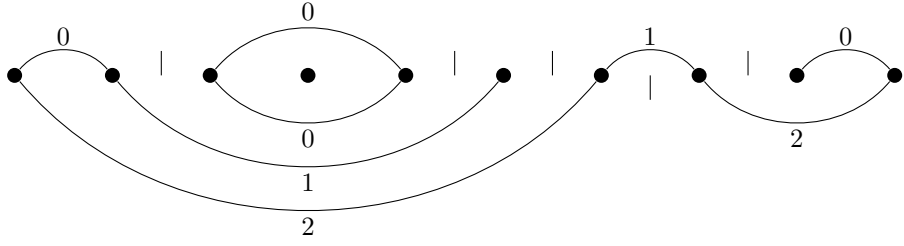
\begin{figure}[h]
        $$\begin{tikzpicture}[scale=1.3]
	\def\Node{\node [circle,  fill, inner sep=2pt]}
	\node at (0,0) {};
     \Node (1) at (0,0) {};
	\Node (2) at (1,0) {};
	\Node (3) at (2,0) {};
	\Node (4) at (3,0) {};
	\Node (5) at (4,0) {};
	\Node (6) at (5,0) {};
    \Node (7) at (6,0) {};
	\Node (8) at (7,0) {};
	\Node (9) at (8,0) {};
	\Node (10) at (9,0) {};
     \draw (1.5, 0)--(1.5, 0.25);
     \draw (4.5, 0)--(4.5, 0.25);
     \draw (5.5, 0)--(5.5, 0.25);
     \draw (7.5, 0.25)--(7.5, 0);
     \draw (6.5, 0)--(6.5, -0.25);
	\draw (3) to[bend left=50] (5);
    \draw (3) to[bend right=50] (5);
    \draw (2) to[bend right=50] (6);
    \draw (2) to[bend right=50] (1);
    \draw (8) to[bend right=50] (7);
	\draw (1) to[bend right=50] (7);
    \draw (9) to[bend left=50] (10);
	\draw (8) to[bend right=50] (10);
    \node at (8.5, 0.4) {$0$};
    \node at (8, -0.65) {$2$};
    \node at (6.5, 0.4) {$1$};
    \node at (0.5, 0.4) {$0$};
    \node at (3, 0.65) {$0$};
    \node at (3, -0.65) {$0$};
    \node at (3, -1.1) {$1$};
    \node at (3, -1.55) {$2$};
\end{tikzpicture}$$
        \caption{$\mathrm{M}_\mathfrak{n}(\mathfrak{s}_1)$ for $\mathfrak{s}_1=\mathfrak{p}\frac{2|3|1|2|2}{7|3}$ or $\mathfrak{p}^A\frac{2|3|1|2|2}{7|3}$}\label{fig:Mns1}
    \end{figure}
\end{ex}

\begin{ex}\label{ex:Mn2}
    For $\mathfrak{s}_2=\mathfrak{p}\frac{3|3|5|2}{6|2|1|2|2}$ or $\mathfrak{p}^A\frac{3|3|5|2}{6|2|1|2|2}$, the meander $\mathrm{M}_\mathfrak{n}(\mathfrak{s}_2)$ is illustrated in Figure~\ref{fig:Mns2}.
    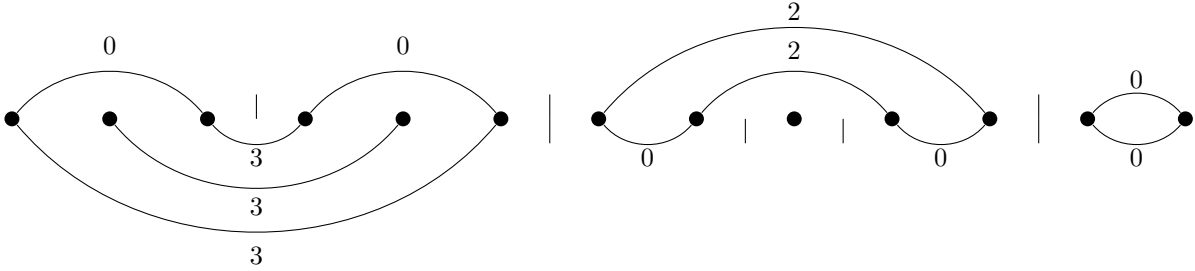
\begin{figure}[h]
        $$\begin{tikzpicture}[scale=1.3]
	\def\Node{\node [circle,  fill, inner sep=2pt]}
	\node at (0,0) {};
     \Node (1) at (0,0) {};
	\Node (2) at (1,0) {};
	\Node (3) at (2,0) {};
	\Node (4) at (3,0) {};
	\Node (5) at (4,0) {};
	\Node (6) at (5,0) {};
    \Node (7) at (6,0) {};
	\Node (8) at (7,0) {};
	\Node (9) at (8,0) {};
	\Node (10) at (9,0) {};
    \Node (11) at (10,0) {};
    \Node (12) at (11,0) {};
    \Node (13) at (12,0) {};
    \draw (2.5, 0)--(2.5, 0.25);
    \draw (5.5, 0)--(5.5, 0.25);
    \draw (5.5, 0)--(5.5, -0.25);
    \draw (7.5, 0)--(7.5, -0.25);
    \draw (8.5, 0)--(8.5, -0.25);
    \draw (10.5, 0)--(10.5, 0.25);
    \draw (10.5, 0)--(10.5, -0.25);
    \draw (1) to[bend left=50] (3);
    \draw (4) to[bend left=50] (6);
    \draw (7) to[bend left=50] (11);
    \draw (8) to[bend left=50] (10);
    \draw (12) to[bend left=50] (13);
    \draw (1) to[bend right=50] (6);
    \draw (2) to[bend right=50] (5);
    \draw (3) to[bend right=50] (4);
    \draw (12) to[bend right=50] (13);
    \draw (7) to[bend right=50] (8);
    \draw (10) to[bend right=50] (11);
    \node at (1,0.75) {$0$};
    \node at (4,0.75) {$0$};
    \node at (2.5,-1.4) {$3$};
    \node at (2.5,-0.9) {$3$};
    \node at (2.5,-0.4) {$3$};
    \node at (6.5,-0.4) {$0$};
    \node at (9.5,-0.4) {$0$};
    \node at (11.5,-0.4) {$0$};
    \node at (11.5,0.4) {$0$};
    \node at (8,0.7) {$2$};
    \node at (8,1.1) {$2$};
\end{tikzpicture}$$
        \caption{$\mathrm{M}_\mathfrak{n}(\mathfrak{s}_2)$ for $\mathfrak{s}_2=\mathfrak{p}\frac{3|3|5|2}{6|2|1|2|2}$ or $\mathfrak{p}^A\frac{3|3|5|2}{6|2|1|2|2}$}\label{fig:Mns2}
    \end{figure}
\end{ex}

Extending language from the case of $\mathrm{M}(\mathfrak{s})$, we define the \demph{central components} of $\mathrm{M}_\mathfrak{n}(\mathfrak{s})$ to be those of $\mathrm{M}(\mathfrak{s})$. In addition, we use the notation $\mathrm{Cen}(\mathfrak{s})$ for the collection of central components of $\mathrm{M}_\mathfrak{n}(\mathfrak{s})$. Finally, to state the main result of this section, for a weighted edge $e$ in $\mathrm{M}_\mathfrak{n}(\mathfrak{s})$, we let $\mathrm{wt}(e)\in\mathbb{Z}_{\ge 0}$ denote the weight assigned.

\begin{theorem}\label{thm:indmain}
    Let $\mathfrak{s}$ be a (type-A) seaweed algebra and $\mathrm{E}(\mathfrak{s})$ denote the collection of edges of $\mathrm{M}_\mathfrak{n}(\mathfrak{s})$. If $\mathfrak{s}$ is a seaweed algebra, then $$\ind\mathfrak{n}(\mathfrak{s})=|\mathrm{Cen}(\mathfrak{s})|+\sum_{e\in \mathrm{E}(\mathfrak{s})}\mathrm{wt}(e);$$ otherwise, if $\mathfrak{s}$ is a type-A seaweed algebra, then $$\ind\mathfrak{n}(\mathfrak{s})=|\mathrm{Cen}(\mathfrak{s})|-1+\sum_{e\in \mathrm{E}(\mathfrak{s})}\mathrm{wt}(e).$$
\end{theorem}

\begin{remark}
    Note that one can compute $|\mathrm{Cen}(\mathfrak{s})|$ combinatorially from $\mathrm{M}_\mathfrak{n}(\mathfrak{s})$ using Proposition~\ref{prop:cccom}. For another approach, it is straightforward to verify that $|\mathrm{Cen}(\mathfrak{s})|$ corresponds to one more than the number of pairs of vertical lines lying above and below the line $L$ in $\mathrm{M}_\mathfrak{n}(\mathfrak{s})$ which occur between the same pair of consecutive vertices $v_i$ and $v_{i+1}$.
\end{remark}

\begin{ex}
    For the seaweed algebra $\mathfrak{s}_1$ of Example~\ref{ex:Mn1}, we have $$|\mathrm{Cen}(\mathfrak{s}_1)|=1\qquad\text{and}\qquad\sum_{e\in \mathrm{E}(\mathfrak{s}_1)}\mathrm{wt}(e)=6,$$ so that $\ind\mathfrak{n}(\mathfrak{s}_1)=1+6=7$. On the other hand,  for the seaweed algebra $\mathfrak{s}_2$ of Example~\ref{ex:Mn2}, we have $$|\mathrm{Cen}(\mathfrak{s}_2)|=3\qquad\text{and}\qquad\sum_{e\in \mathrm{E}(\mathfrak{s}_2)}\mathrm{wt}(e)=13,$$ so that $\ind\mathfrak{n}(\mathfrak{s}_2)=3+13=16$. To find the index of the type-A versions of $\mathfrak{s}_1$ and $\mathfrak{s}_2$, subtract one from the values found above.
\end{ex}

Since $\ind\mathfrak{g}_1\oplus\mathfrak{g}_2=\ind\mathfrak{g}_1+\ind\mathfrak{g}_2$ for Lie algebras $\mathfrak{g}_1$ and $\mathfrak{g}_2$, we break the proof of Theorem~\ref{thm:indmain} into two pieces: considering first $\ind\mathrm{Z}(\mathfrak{s})$ and then $\ind\mathfrak{g}^\prec(\mathcal{P}_\mathfrak{s})$. For $\ind\mathrm{Z}(\mathfrak{s})$, note that, considering the definition of index, it follows that $\ind\mathrm{Z}(\mathfrak{s})=\dim\mathrm{Z}(\mathfrak{s})$. Thus, we get the following as an immediate consequence of Theorem~\ref{thm:center}.

\begin{lemma}\label{lem:indcent}
    If $\mathfrak{s}$ is a seaweed algebra, then $$\ind\mathrm{Z}(\mathfrak{s})=|\mathrm{Cen}(\mathfrak{s})|.$$ Similarly, if $\mathfrak{s}$ is a type-A seaweed algebra, then $$\ind\mathrm{Z}(\mathfrak{s})=|\mathrm{Cen}(\mathfrak{s})|-1.$$
\end{lemma}

Shifting focus to $\ind\mathfrak{g}^\prec(\mathcal{P}_\mathfrak{s})$, it suffices to restrict attention to $\mathfrak{s}$ a seaweed algebra. Our main result concerning $\ind\mathfrak{g}^\prec(\mathcal{P}_\mathfrak{s})$ is as follows.

\begin{theorem}\label{thm:indpo}
    Let $\mathfrak{s}$ be a seaweed algebra. If $\mathrm{E}(\mathfrak{s})$ denotes the collection of edges in $\mathrm{M}_\mathfrak{n}(\mathfrak{s})$, then $$\ind\mathfrak{g}^\prec(\mathcal{P}_\mathfrak{s})=\sum_{e\in \mathrm{E}(\mathfrak{s})}\mathrm{wt}(e).$$
\end{theorem}

For the proof of Theorem~\ref{thm:indpo} we require the following results. First, Lemma~\ref{lem:recur} provides a recursive formula for the index of nilpotent Lie poset algebras of the form $\mathfrak{g}^\prec(a_1,\cdots,a_m)$. Next, utilizing Lemma~\ref{lem:recur}, it is shown that Theorem~\ref{thm:indpo} holds in the case that $\mathfrak{s}=\mathfrak{p}\frac{a_1|\cdots|a_m}{N}$ or $\mathfrak{p}\frac{N}{a_m|\cdots|a_1}$ in Proposition~\ref{prop:parind}. Finally, Lemma~\ref{lem:paraenough} identifies a poset decomposition which is used to show that Theorem~\ref{thm:indpo} follows as a consequence of Proposition~\ref{prop:parind}.

\begin{lemma}\label{lem:recur}
    Let $\mathfrak{g}_1=\mathfrak{g}^\prec(a_1,\cdots,a_m)$ for $m\ge 2$. Moreover, when $m>2$, let $$\mathfrak{g}_2=\begin{cases}
        \mathfrak{g}^\prec(a_2,\cdots,a_{m-1}), & a_1=a_m\\
        \\
        \mathfrak{g}^\prec(a_1-a_m,\cdots,a_{m-1}), & a_1>a_m \\
        \\
        \mathfrak{g}^\prec(a_2,\cdots,a_m-a_1), & a_1<a_m.
    \end{cases}$$
    Then $\ind\mathfrak{g}_1=a_1a_2$ when $m=2$, and $\ind\mathfrak{g}_1=\ind\mathfrak{g}_2+a_1a_m$ otherwise.
\end{lemma}
\begin{proof}
    If $m=2$ so that $\mathfrak{g}_1=\mathfrak{g}^\prec(a_1,a_2)$, then, applying Theorem~\ref{thm:indexnil}, we find that $\ind \mathfrak{g}_1=a_1a_2$, as desired. Now, assume that $m\ge 3$. There are two cases to consider: $a_1=a_m$ and $a_1\neq a_m$. We consider the case where $a_1\neq a_m$, as the proof for $a_1=a_m$ is similar and simpler. Without loss of generality, assume that $a_1>a_m$. Let $\mathcal{P}_1=\mathcal{P}(a_1,\cdots,a_m)$ and $\mathcal{P}_2=\mathcal{P}(a _1-a_m,\cdots,a_{m-1})$. Then setting $$\mathcal{Q}_1=\{a_m+1,\hdots,a_1\},\quad\mathcal{Q}_2=\left\{a_1+1,\hdots,\sum_{i=1}^{m-2}a_i\right\},\quad\text{and}\quad\mathcal{Q}_3=\left\{1+\sum_{i=1}^{m-2}a_i,\hdots,\sum_{i=1}^{m-1}a_i\right\},$$ we have that $\mathcal{P}_1$ restricted to $\mathcal{Q}_1\cup\mathcal{Q}_2\cup\mathcal{Q}_3$ is isomorphic to $\mathcal{P}_2$ with $\mathcal{Q}_2$ corresponding to $\mathcal{P}_2\backslash\mathrm{Ext}(\mathcal{P}_2)$. Now, note that $$|\mathrm{Rel}(\mathcal{P}_1)|=|\mathrm{Rel}(\mathcal{P}_2)|+a_m(|\mathcal{Q}_2|+|\mathcal{Q}_3|)+a_m|\mathcal{P}_2|+a_m^2$$ and, letting $v(\mathcal{P},p)=\min(D(\mathcal{P},p),U(\mathcal{P},p))$ for $p\in\mathcal{P}$, we have
    \begin{align*}
        2\sum_{p\in\mathcal{P}_1\backslash\mathrm{Ext}(\mathcal{P}_1)}v(\mathcal{P}_1,p)&=2\sum_{p\in\mathcal{Q}_3}v(\mathcal{P}_1,p)+2\sum_{p\in\mathcal{Q}_2}v(\mathcal{P}_1,p)\\
        &=2a_m|\mathcal{Q}_3|+2\sum_{p\in \mathcal{P}_2\backslash\mathrm{Ext}(\mathcal{P}_2)}(v(\mathcal{P}_2,p)+a_m)\\
        &=2a_m|\mathcal{Q}_3|+2a_m|\mathcal{Q}_2|+2\sum_{p\in \mathcal{P}_2\backslash\mathrm{Ext}(\mathcal{P}_2)}v(\mathcal{P}_2,p)\\
        &=2a_m(|\mathcal{Q}_2|+|\mathcal{Q}_3|)+2\sum_{p\in \mathcal{P}_2\backslash\mathrm{Ext}(\mathcal{P}_2)}v(\mathcal{P}_2,p).
    \end{align*}
    Thus, applying Theorem~\ref{thm:indexnil}, we have that
    \begin{align*}
        \ind\mathfrak{g}_1&=|\mathrm{Rel}(\mathcal{P}_1)|-2\sum_{p\in\mathcal{P}_1\backslash\mathrm{Ext}(\mathcal{P}_1)}v(\mathcal{P}_1,p)\\
        &=|\mathrm{Rel}(\mathcal{P}_2)|+a_m(|\mathcal{Q}_2|+|\mathcal{Q}_3|)+a_m|\mathcal{P}_2|+a_m^2-\left(2a_m(|\mathcal{Q}_2|+|\mathcal{Q}_3|)+2\sum_{p\in \mathcal{P}_2\backslash\mathrm{Ext}(\mathcal{P}_2)}v(\mathcal{P}_2,p)\right)\\
        &=\left(|\mathrm{Rel}(\mathcal{P}_2)|-2\sum_{p\in \mathcal{P}_2\backslash\mathrm{Ext}(\mathcal{P}_2)}v(\mathcal{P}_2,p)\right)+a_m(|\mathcal{P}_2|-|\mathcal{Q}_2|-|\mathcal{Q}_3|)+a_m^2\\
        &=\ind\mathfrak{g}_2+a_m|\mathcal{Q}_1|+a_m^2=\ind\mathfrak{g}_2+a_m(a_1-a_m)+a_m^2=\ind\mathfrak{g}_2+a_1a_m
    \end{align*}
    as desired. 
\end{proof}

\begin{prop}\label{prop:parind}
    Let $\mathfrak{s}=\mathfrak{p}\frac{a_1|\cdots|a_m}{N}$ or $\mathfrak{p}\frac{N}{a_m|\cdots|a_1}$. If $\mathrm{E}(\mathfrak{s})$ denotes the collection of edges in $\mathrm{M}_\mathfrak{n}(\mathfrak{s})$, then $$\ind\mathfrak{g}^\prec(\mathcal{P}_\mathfrak{s})=\sum_{e\in \mathrm{E}(\mathfrak{s})}\mathrm{wt}(e).$$
\end{prop}
\begin{proof}
    Without loss of generality, assume that $\mathfrak{s}=\mathfrak{p}\frac{a_1|\cdots|a_m}{N}$. We prove the result holds by induction on $m$. For the base case, note that when $m=1$ we have that $\mathfrak{g}^\prec(\mathcal{P}_\mathfrak{s})=\{0\}$ and all edges of $\mathrm{M}_\mathfrak{n}(\mathfrak{s})$ have weight 0, i.e., $$\ind\mathfrak{g}^\prec(\mathcal{P}_\mathfrak{s})=0=\sum_{e\in \mathrm{E}(\mathfrak{s})}\mathrm{wt}(e).$$ Assume the result holds for $m-1\ge 0$. Take $\mathfrak{s}=\mathfrak{p}\frac{a_1|\cdots|a_m}{N}$ so that $\mathfrak{g}^{\prec}(\mathcal{P}_\mathfrak{s})\cong\mathfrak{g}^{\prec}(a_1,\hdots,a_m)$ with $\sum_{i=1}^ma_i=N$ by Corollary~\ref{cor:parbposet}. There are two cases. In both, we ignore vertex labels of meanders under discussion.
    \bigskip

    \noindent
    \textbf{Case 1:} $a_1=a_m$. Assume that $m>2$, the argument for $m=2$ being similar and simpler. In this case, by construction, the $a_1$ bottom edges connecting the first to the last $a_1$ vertices in $\mathrm{M}_\mathfrak{n}(\mathfrak{s})$ are all weighted $a_1$. Moreover, the removal of the vertices connected by these edges in $\mathrm{M}_\mathfrak{n}(\mathfrak{s})$ along with all associated top edges results in $\mathrm{M}_\mathfrak{n}(\mathfrak{s}')$ where $\mathfrak{s}'\cong \mathfrak{p}\frac{a_2|\cdots|a_{m-1}}{N-2a_1}$. Note that $\mathfrak{g}^{\prec}(\mathcal{P}_{\mathfrak{s}'})\cong\mathfrak{g}^\prec(a_2,\hdots,a_{m-1})$ by Corollary~\ref{cor:parbposet}. Thus, applying our induction hypothesis along with Lemma~\ref{lem:recur}, we find that $$\sum_{e\in \mathrm{E}(\mathfrak{s})}\mathrm{wt}(e)=a_1^2+\sum_{e\in \mathrm{E}(\mathfrak{s}')}\mathrm{wt}(e)\\
        =a_1^2+\ind\mathfrak{g}^\prec(\mathcal{P}_{\mathfrak{s}'})\\
        =\ind\mathfrak{g}^\prec(\mathcal{P}_{\mathfrak{s}}),$$
    as desired.
    \bigskip

    \noindent
    \textbf{Case 2:} $a_1>a_m$ or $a_1<a_m$. Without loss of generality, assume that $a_1>a_m$. In addition, assume that $m>2$, the argument for $m=2$ being similar and simpler. In this case, by construction, the $a_m$ bottom edges connecting the first to the last $a_m$ vertices in $\mathrm{M}_\mathfrak{n}(\mathfrak{s})$ are all weighted $a_1$. Moreover, we can form $\mathrm{M}_\mathfrak{n}(\mathfrak{s}')$ where $\mathfrak{s}'\cong \mathfrak{p}\frac{a_1-a_m|\cdots|a_{m-1}}{N-2a_m}$ from $\mathrm{M}_\mathfrak{n}(\mathfrak{s})$ by
    \begin{itemize}
        \item removing the aforementioned $a_m$ bottom edges along with any top edges connecting the first $a_1$ vertices as well as the last $a_m$,
        \item removing the first and last $a_m$ vertices,
        \item and then adding a top edge weighted by $0$ between the $i^{th}$ and $(a_1-a_m-i+1)^{th}$ vertex for $1\le i\le \lfloor\frac{a_1-a_m}{2}\rfloor$ assuming that the vertices are distinct.
    \end{itemize}
    Note that $\mathfrak{g}^{\prec}(\mathcal{P}_{\mathfrak{s}'})\cong\mathfrak{g}^\prec(a_1-a_m,\hdots,a_{m-1})$ by Corollary~\ref{cor:parbposet}. Thus, applying our induction hypothesis along with Lemma~\ref{lem:recur}, we find that $$\sum_{e\in \mathrm{E}(\mathfrak{s})}\mathrm{wt}(e)=a_1a_m+\sum_{e\in \mathrm{E}(\mathfrak{s}')}\mathrm{wt}(e)\\
        =a_1a_m+\ind\mathfrak{g}^\prec(\mathcal{P}_{\mathfrak{s}'})\\
        =\ind\mathfrak{g}^\prec(\mathcal{P}_{\mathfrak{s}}),$$
    as desired.
\end{proof}

\begin{lemma}\label{lem:paraenough}
    Let $\mathcal{P}$ be a poset and $\mathcal{P}_1,\mathcal{P}_2\subseteq\mathcal{P}$ satisfy
    \begin{enumerate}
        \item[$(1)$] $\mathcal{P}_1\cup\mathcal{P}_2=\mathcal{P}$;
        \item[$(2)$] $\mathcal{P}_1\cap\mathcal{P}_2\subseteq\mathrm{Min}(\mathcal{P})$ or  $\mathcal{P}_1\cap\mathcal{P}_2\subseteq\mathrm{Max}(\mathcal{P})$; and
        \item[$(3)$] if $x\in\mathcal{P}_1$ and $y\in\mathcal{P}_2$ satisfy $x\prec y$ or $y\prec x$, then either $x\in\mathcal{P}_1\cap\mathcal{P}_2$ or $y\in\mathcal{P}_1\cap\mathcal{P}_2$.
    \end{enumerate}
    Then $\mathfrak{g}^{\prec}(\mathcal{P})\cong \mathfrak{g}^{\prec}(\mathcal{P}_1)\oplus \mathfrak{g}^{\prec}(\mathcal{P}_2)$.
\end{lemma}
\begin{proof}
    Evidently, the result holds at the vector space level. As for the Lie structure, note that, by assumptions $(2)$ and $(3)$, if $E_{x_1,y_1}\in \mathfrak{g}^{\prec}(\mathcal{P}_1)$ and $E_{x_2,y_2}\in \mathfrak{g}^{\prec}(\mathcal{P}_2)$, then $\{x_1,y_1\}\cap\{x_2,y_2\}=0$ or $1$. Moreover, if $\{x_1,y_1\}\cap\{x_2,y_2\}=1$, then either $x_1=x_2$ or $y_1=y_2$. Either way, for all $E_{x_1,y_1}\in \mathfrak{g}^{\prec}(\mathcal{P}_1)$ and $E_{x_2,y_2}\in \mathfrak{g}^{\prec}(\mathcal{P}_2)$, it follows that $[E_{x_1,y_1},E_{x_2,y_2}]=0$.
\end{proof}

With Proposition~\ref{prop:parind} and Lemma~\ref{lem:paraenough}, we can now prove Theorem~\ref{thm:indpo} and, consequently, Theorem~\ref{thm:indmain}.

\begin{proof}[Proof of Theorem~\ref{thm:indpo}]
    As noted in Section~\ref{sec:posetalg}, identifying $i$ with $v_i$ for $i\in\mathcal{P}_\mathfrak{s}$, the connected components of $\mathcal{P}_\mathfrak{s}$ are in bijection with the central components of $\mathrm{M}_\mathfrak{n}(\mathfrak{s})$. Consequently, letting $\mathcal{Q}_1,\hdots,\mathcal{Q}_n$ denote the connected components of $\mathcal{P}_\mathfrak{s}$, there exist seaweed algebras $\mathfrak{s}_i$ for $i\in [n]$ satisfying $\mathcal{Q}_i\cong\mathcal{P}_{\mathfrak{s}_i}$ and, ignoring vertex labels, $\mathrm{M}_\mathfrak{n}(\mathfrak{s}_i)$ is the central component of $\mathrm{M}_\mathfrak{n}(\mathfrak{s})$ corresponding to $\mathcal{Q}_i$. In particular, for $i\in [n]$, the seaweed algebra $\mathfrak{s}_i$ is defined by the compositions given by the sizes of the top and bottom blocks of the central component associated with $\mathcal{Q}_i$. Thus, applying Lemma~\ref{lem:paraenough}, it follows that $\mathfrak{g}^\prec(\mathcal{P}_\mathfrak{s})\cong \mathfrak{g}^\prec(\mathcal{P}_{\mathfrak{s}_1})\oplus\cdots\oplus \mathfrak{g}^\prec(\mathcal{P}_{\mathfrak{s}_n})$. Moreover, since there are no edges between distinct central components of $\mathrm{M}_\mathfrak{n}(\mathfrak{s})$ and $\ind\mathfrak{g}_1\oplus\mathfrak{g}_2=\ind\mathfrak{g}_1+\ind\mathfrak{g}_2$ for Lie algebras $\mathfrak{g}_1$ and $\mathfrak{g}_2$, we find that it suffices to establish the result for $\mathcal{P}_\mathfrak{s}$ connected. In this case, as discussed in Section~\ref{sec:posetalg}, the poset $\mathcal{P}_\mathfrak{s}$ is either of the form  
    \begin{enumerate}
        \item[$(a)$] $\mathrm{in}(\mathcal{P}_1,\hdots,\mathcal{P}_\ell)$ where $\mathcal{P}_j\cong\mathcal{P}(a_{1,j},\hdots,a_{m_j,j})$ for $j\in [\ell]$, $a_{1,j}=a_{1,j+1}$ for $j\in [\ell-1]$ odd, and $a_{m_j,j}=a_{m_{j+1},j+1}$ for $j\in [\ell-1]$ even; or
        \item[$(b)$] $\mathrm{out}(\mathcal{P}_1,\hdots,\mathcal{P}_\ell)$ where $\mathcal{P}_j\cong\mathcal{P}(a_{1,j},\hdots,a_{m_j,j})$ for $j\in [\ell]$, $a_{1,j}=a_{1,j+1}$ for $j\in [\ell-1]$ even, and $a_{m_j,j}=a_{m_{j+1},j+1}$ for $j\in [\ell-1]$ odd.
    \end{enumerate}
    We prove the result by induction on $\ell$. If $\ell=1$, then $\mathfrak{s}=\mathfrak{p}\frac{a_{1,1}|\cdots|a_{m_1,1}}{N}$ or $\mathfrak{p}\frac{N}{a_{m_1,1}|\cdots|a_{1,1}}$ and the result follows by Proposition~\ref{prop:parind}. So assume that the result holds for $\ell-1\ge 1$. Take $\mathcal{P}_\mathfrak{s}=\mathrm{in}(\mathcal{P}_1,\hdots,\mathcal{P}_\ell)$, the case of $\mathcal{P}_\mathfrak{s}=\mathrm{out}(\mathcal{P}_1,\hdots,\mathcal{P}_\ell)$ following by a similar argument. Considering the definition of $\mathrm{M}_\mathfrak{p}(\mathfrak{s})$, we find that if $\mathfrak{s}=\mathfrak{p}\frac{a_1|\cdots|a_m}{b_1|\cdots|b_n}$, then $b_i=a_{i,1}$ for $i\in [m_1]$ and $a_1=\sum_{i=1}^{m_1}a_{i,1}$. Moreover, letting $\mathfrak{s}_1=\mathfrak{p}\frac{a_1}{a_{1,1}|\cdots|a_{m_1,1}}$ and $\mathfrak{s}_2=\mathfrak{p}\frac{a_{m_1,1}|a_2|\cdots|a_m}{a_{m_1,1}|b_{m_1+1}|\cdots|b_n}$, we have that $\mathcal{P}_{\mathfrak{s}_1}\cong\mathcal{P}_1$ and $\mathcal{P}_{\mathfrak{s}_2}\cong\mathrm{out}(\mathcal{P}_2,\hdots,\mathcal{P}_\ell)$. Now, note that one forms $\mathrm{M}_\mathfrak{n}(\mathfrak{s})$ from $\mathrm{M}_\mathfrak{n}(\mathfrak{s}_1)$ and $\mathrm{M}_\mathfrak{n}(\mathfrak{s}_2)$ by removing the bottom edges connecting the last $a_{m_1,1}$ vertices of $\mathrm{M}_\mathfrak{n}(\mathfrak{s}_1)$ along with the bottom edges connecting the first $a_{m_1,1}$ vertices of $\mathrm{M}_\mathfrak{n}(\mathfrak{s}_2)$, all of which have weight 0, and then identifying these vertices of $\mathrm{M}_\mathfrak{n}(\mathfrak{s}_1)$ and $\mathrm{M}_\mathfrak{n}(\mathfrak{s}_2)$. Consequently, there is a weight-preserving bijection between the edges with nonzero weight in $\mathrm{M}_\mathfrak{n}(\mathfrak{s}_1)$ and $\mathrm{M}_\mathfrak{n}(\mathfrak{s}_2)$ with those in $\mathrm{M}_\mathfrak{n}(\mathfrak{s})$. Thus, since $\ind\mathfrak{g}^\prec(\mathcal{P}_\mathfrak{s})=\ind\mathfrak{g}^\prec(\mathcal{P}_{\mathfrak{s}_1})+\ind\mathfrak{g}^\prec(\mathcal{P}_{\mathfrak{s}_2})$ by Lemma~\ref{lem:paraenough}, applying our induction hypothesis the result follows.
\end{proof}

As corollaries to Theorem~\ref{thm:indmain}, we get the following nilradical analogues of the index formulas for $\mathfrak{p}^A\frac{a|b}{N}$ as well as $\mathfrak{p}^A\frac{a|b}{c|d}$, and $\mathfrak{p}^A\frac{a|b|c}{N}$ provided by Theorems 6 and 8, respectively, of \cite{Coll2}. The details of the proofs are left to the interested reader.

\begin{corollary}~\label{lem:mpind}
    \begin{enumerate}
        \item[$(a)$] If $\mathfrak{s}=\mathfrak{p}^A\frac{a|b}{N}$ or $\mathfrak{p}^A\frac{a|b}{c|d}$, then $\ind \mathfrak{n}(\mathfrak{s})=ab$.
        \item[$(b)$] If $\mathfrak{s}=\mathfrak{p}^A\frac{a|b|c}{N}$, then $\ind \mathfrak{n}(\mathfrak{s})=ac+b|a-c|$.
    \end{enumerate}
    
\end{corollary}

\begin{remark}
    Note that for a seaweed algebra $\mathfrak{s}$, there is no general relationship between $\ind\mathfrak{n}(\mathfrak{s})$ and $\ind\mathfrak{s}$. To see this, taking $\mathfrak{s}_1=\mathfrak{p}\frac{1|2|1}{4}$ one has $2=\ind\mathfrak{n}(\mathfrak{s}_1)<\ind\mathfrak{s}_1=3,$ while taking $\mathfrak{s}_2=\mathfrak{p}\frac{1|2}{2|1}$ one has $3=\ind\mathfrak{n}(\mathfrak{s}_2)>\ind\mathfrak{s}_2=2.$ Evidently, the same applies to type-A seaweed algebras. 
\end{remark}

Recalling from Section~\ref{sec:prelim} that $\mathrm{E}_1(\mathfrak{s})$ denotes the collection of edges in $\mathrm{M}(\mathfrak{s})$ for which no other edge connects the same two vertices, with respect to the original meander $\mathrm{M}(\mathfrak{s})$, we have the following.

\begin{theorem}\label{thm:Ms}
    If $\mathfrak{s}$ is a seaweed algebra, then $$\ind\mathfrak{n}(\mathfrak{s})\ge |\mathrm{E}_1(\mathfrak{s})|+|\mathrm{Cen}(\mathfrak{s})|.$$ Moreover, if $\mathfrak{s}$ is a type-A seaweed algebra, then $$\ind\mathfrak{n}(\mathfrak{s})\ge |\mathrm{E}_1(\mathfrak{s})|+|\mathrm{Cen}(\mathfrak{s})|-1.$$
\end{theorem}
\begin{proof}
    Considering Theorem~\ref{thm:indmain}, it suffices to consider the case where $\mathfrak{s}$ is a seaweed algebra. Moreover, since $\mathfrak{n}(\mathfrak{s})\cong \mathfrak{g}^\prec(\mathcal{P}_\mathfrak{s})\oplus\mathrm{Z}(\mathfrak{s})$ by Theorem~\ref{thm:nilposet} and $\ind\mathrm{Z}(\mathfrak{s})=|\mathrm{Cen}(\mathfrak{s})|$, we need only show that $\ind\mathfrak{g}^\prec(\mathcal{P}_\mathfrak{s})\ge |\mathrm{E}_1(\mathfrak{s})|$. Arguing as in the proof of Theorem~\ref{thm:indpo}, if $\mathcal{P}_\mathfrak{s}$ has connected components $\mathcal{Q}_1,\hdots,\mathcal{Q}_n$, then there exist seaweed algebras $\mathfrak{s}_1,\hdots,\mathfrak{s}_n$ such that $\mathcal{Q}_i\cong\mathcal{P}_{\mathfrak{s}_i}$ for $i\in [n]$. Moreover, the central components of $\mathrm{M}(\mathfrak{s})$ are, ignoring vertex labels, exactly the $\mathrm{M}(\mathfrak{s}_i)$ for $i\in [n]$. Consequently, simplifying matters further, applying Lemma~\ref{lem:paraenough}, we need only consider the case where $\mathcal{P}_\mathfrak{s}$ is connected; that is, we may assume that  $\mathcal{P}_\mathfrak{s}$ is either of the form  
    \begin{enumerate}
        \item[$(a)$] $\mathrm{in}(\mathcal{P}_1,\hdots,\mathcal{P}_\ell)$ where $\mathcal{P}_j\cong\mathcal{P}(a_{1,j},\hdots,a_{m_j,j})$ for $j\in [\ell]$, $a_{1,j}=a_{1,j+1}$ for $j\in [\ell-1]$ odd, and $a_{m_j,j}=a_{m_{j+1},j+1}$ for $j\in [\ell-1]$ even; or
        \item[$(b)$] $\mathrm{out}(\mathcal{P}_1,\hdots,\mathcal{P}_\ell)$ where $\mathcal{P}_j\cong\mathcal{P}(a_{1,j},\hdots,a_{m_j,j})$ for $j\in [\ell]$, $a_{1,j}=a_{1,j+1}$ for $j\in [\ell-1]$ even, and $a_{m_j,j}=a_{m_{j+1},j+1}$ for $j\in [\ell-1]$ odd.
    \end{enumerate}
    As in the proof of Theorem~\ref{thm:indpo}, we prove the result by induction on $\ell$. If $\ell=1$, then setting $m=m_1$, $a_i=a_{i,1}$ for $i\in [m]$, and $N=\sum_{i=1}^ma_i$, we have that $\mathfrak{g}^\prec(\mathcal{P}_\mathfrak{s})\cong\mathfrak{g}^{\prec}(a_1,\hdots,a_m)$ and either $\mathfrak{s}=\mathfrak{p}\frac{a_1|\cdots|a_m}{N}$ or $\mathfrak{p}\frac{N}{a_m|\cdots|a_1}$. Without loss of generality, assume that $\mathfrak{s}=\mathfrak{p}\frac{a_1|\cdots|a_m}{N}$. We show that the base case holds by induction on $m$. Note that if $m=1$, then $\ind\mathfrak{g}^\prec(\mathcal{P}_\mathfrak{s})=0$ and $\mathrm{E}_1(\mathfrak{s})$ is empty. For $m=2$, applying Lemma~\ref{lem:mpind}, we have that $$\ind\mathfrak{g}^\prec(\mathcal{P}_\mathfrak{s})=\ind\mathfrak{g}^\prec(a_1,a_2)=a_1a_2\ge\left\lfloor\frac{a_1}{2}\right\rfloor+\left\lfloor\frac{a_2}{2}\right\rfloor+\left\lfloor\frac{a_1+a_2}{2}\right\rfloor=E_1(\mathfrak{s}).$$ Assume the result holds for $m-1\ge 2$. Let $\mathfrak{g}_1=\mathfrak{g}^\prec(\mathcal{P}_\mathfrak{s})=\mathfrak{g}^\prec(a_1,\hdots,a_m)$. Without loss of generality, assume that $a_1\ge a_m$. Define $$\mathfrak{s}_2=\begin{cases}
        \mathfrak{p}\frac{a_2|\cdots|a_{m-1}}{N-2a_1}, & a_1=a_m \\
        \\
        \mathfrak{p}\frac{a_1-a_m|a_2|\cdots|a_{m-1}}{N-2a_m}, & a_1>a_m
    \end{cases}\qquad\text{and}\qquad \mathfrak{g}_2=\begin{cases}
        \mathfrak{g}^\prec(a_2,\hdots,a_{m-1}), & a_1=a_m \\
        \\
        \mathfrak{g}^\prec(a_1-a_m,a_2,\hdots,a_{m-1}), & a_1>a_m.
    \end{cases}$$
    Note that $$\mathrm{E}_1(\mathfrak{s}_1)=\mathrm{E}_1(\mathfrak{s}_2)+\left\lfloor\frac{a_1}{2}\right\rfloor+\left\lfloor\frac{a_m}{2}\right\rfloor+a_m.$$ Now, applying Lemma~\ref{lem:recur} along with the induction hypothesis on $m$, we have
    \begin{align*}
        \ind\mathfrak{g}_1=\ind\mathfrak{g}_2+a_1a_m\ge \mathrm{E}_1(\mathfrak{s}_2)+a_1a_m&\ge \mathrm{E}_1(\mathfrak{s}_2)+\left\lfloor\frac{a_1}{2}\right\rfloor+\left\lfloor\frac{a_m}{2}\right\rfloor+\left\lfloor\frac{a_1+a_m}{2}\right\rfloor\\
        &\ge \mathrm{E}_1(\mathfrak{s}_2)+\left\lfloor\frac{a_1}{2}\right\rfloor+\left\lfloor\frac{a_m}{2}\right\rfloor+a_m=\mathrm{E}_1(\mathfrak{s}_1),
    \end{align*}
    as desired. Thus, the base case holds. Now, assume that the result holds for $\ell-1\ge 1$. Take $\mathcal{P}_\mathfrak{s}=\mathrm{in}(\mathcal{P}_1,\hdots,\mathcal{P}_\ell)$, the case of $\mathcal{P}_\mathfrak{s}=\mathrm{out}(\mathcal{P}_1,\hdots,\mathcal{P}_\ell)$ following by a similar argument. Arguing as in the proof of Theorem~\ref{thm:indpo}, there exist seaweed algebras $\mathfrak{s}_1$ and $\mathfrak{s}_2$ such that $\mathcal{P}_{\mathfrak{s}_1}\cong\mathcal{P}_1$ and $\mathcal{P}_{\mathfrak{s}_2}\cong\mathrm{out}(\mathcal{P}_2,\hdots,\mathcal{P}_\ell)$. Moreover, $\mathrm{M}(\mathfrak{s})$ is formed from $\mathrm{M}(\mathfrak{s}_1)$ and $\mathrm{M}(\mathfrak{s}_2)$ by removing the edges below connecting the last $a_{m_1,1}$ vertices of $\mathrm{M}(\mathfrak{s}_1)$ along with the edges below connecting the first $a_{m_1,1}$ vertices of $\mathrm{M}(\mathfrak{s}_2)$ and then identifying these vertices of $\mathrm{M}(\mathfrak{s}_1)$ and $\mathrm{M}(\mathfrak{s}_2)$. Consequently, $|\mathrm{E}_1(\mathfrak{s})|\le |\mathrm{E}_1(\mathfrak{s}_1)|+|\mathrm{E}_1(\mathfrak{s}_2)|$. Therefore, since $\ind\mathfrak{g}^\prec(\mathcal{P}_\mathfrak{s})=\ind\mathfrak{g}^\prec(\mathcal{P}_{\mathfrak{s}_1})+\ind\mathfrak{g}^\prec(\mathcal{P}_{\mathfrak{s}_2})$ by Lemma~\ref{lem:mpind}, applying the induction hypothesis on $\ell$ the result follows.
\end{proof}

Interestingly, the bound given in Theorem~\ref{thm:Ms} can be tight. Examples are provided by the proposition below.

\begin{prop}
    Let $\mathfrak{s}=\mathfrak{p}\frac{a_1|\hdots|a_n}{b_1|\hdots|b_m}$ or $\mathfrak{p}\frac{a_1|\hdots|a_n}{N}$ where $a_i,b_j\in [2]$ for $i\in [n]$ and $j\in [m]$. Then $$\ind\mathfrak{n}(\mathfrak{s})= |\mathrm{E}_1(\mathfrak{s})|+|\mathrm{Cen}(\mathfrak{s})|.$$ In the case that $\mathfrak{s}$ is a type-A seaweed algebra, we have $$\ind\mathfrak{n}(\mathfrak{s})= |\mathrm{E}_1(\mathfrak{s})|+|\mathrm{Cen}(\mathfrak{s})|-1.$$
\end{prop}
\begin{proof}
    Arguing as in the proof of Theorem~\ref{thm:Ms}, it suffices to establish the result when $\mathfrak{s}$ is a seaweed algebra. Moreover, we need only show that $\ind\mathfrak{g}^\prec(\mathcal{P}_\mathfrak{s})=|\mathrm{E}_1(\mathfrak{s})|$; that is, by Theorem~\ref{thm:indpo}, the sums of the weights of edges in $\mathrm{M}_\mathfrak{n}(\mathfrak{s})$ is equal to $|\mathrm{E}_1(\mathfrak{s})|$. We consider the two cases separately. To start, assume that $\mathfrak{s}=\mathfrak{p}\frac{a_1|\hdots|a_n}{b_1|\hdots|b_m}$ where $a_i,b_j\in [2]$ for $i\in [n]$ and $j\in [m]$. Note that, evidently, each top and bottom block of $\mathrm{M}(\mathfrak{s})$ has size one or two. We claim that all edges of $\mathrm{E}_1(\mathfrak{s})$ have a weight of one in $\mathrm{M}_\mathfrak{n}(\mathfrak{s})$ and all others have weight zero. To see this, given the sizes of blocks, initially it seems that potential weights for edges are zero, one, or two. Now, having an edge of weight two would require the existence of vertices $v_i$ and $v_j$ in $\mathrm{M}_\mathfrak{n}(\mathfrak{s})$ with $i<i+1<j$ for which there exists an edge joining $v_i$ and $v_{i+1}$, as well as an edge joining $v_i$ and $v_j$; but this would imply there exists $a_i$ for $i\in [n]$ with $a_i>2$ or $b_j$ for $j\in [m]$ with $b_j>2$, which is a contradiction. Thus, all edges of $\mathrm{M}_\mathfrak{n}(\mathfrak{s})$ are weighted by zero or one. Now, given that all edges of $\mathrm{M}_\mathfrak{n}(\mathfrak{s})$ connect adjacent vertices under our restrictions, it follows that an edge is weighted by zero if and only if it does not belong to $\mathrm{E}_1(\mathfrak{s})$. Thus, the claim and, consequently, the result, follow.

    Finally, assume that $\mathfrak{s}=\mathfrak{p}\frac{a_1|\hdots|a_n}{N}$ where $a_i\in [2]$ for $i\in [n]$. Considering the definition of $\mathrm{M}_\mathfrak{n}(\mathfrak{s})$, all edges with nonzero weight are bottom edges. Moreover, since the possible sizes of top blocks are one or two, it follows that all such edges have weight zero, one, or two. Note that having weight zero must correspond to an edge not belonging to $\mathrm{E}_1(\mathfrak{s})$. We claim that each bottom edge $e$ with weight two can be paired uniquely with a top edge belonging to $\mathrm{E}_1(\mathfrak{s})$. To see this, let's say that such a bottom edge $e_1$ connects vertices $v_i$ and $v_j$ with $i<j$. There are four possibilities:
    \begin{itemize}
        \item[$(1)$] $v_i$ is adjacent to $v_{i-1}$ and $v_j$ to $v_{j+1}$ by top edges, say $e_2$ and $e_3$, respectively. Note that, $v_{i-1}$ and $v_{j+1}$ are connected by a bottom edge. In this case, pair $e_1$ with $e_3$; 
        \item[$(2)$] $v_i$ is adjacent to $v_{i+1}$ and $v_j$ to $v_{j-1}$ by top edges, say $e_2$ and $e_3$, respectively. Note that, $v_{i+1}$ and $v_{j-1}$ are connected by a bottom edge. In this case, pair $e_1$ with $e_2$;
        \item[$(3)$] $v_i$ is adjacent to $v_{i+1}$ by a top edge, say $e_2$, and $v_j$ is either adjacent to $v_{j+1}$ by a top edge, say $e_3$, or adjacent to no other vertex by a top edge. In this case, pair $e_1$ with $e_2$; and
        \item[$(4)$] $v_j$ is adjacent to $v_{j-1}$ by a top edge, say $e_2$, and $v_i$ is either adjacent to $v_{i-1}$ by a top edge, say $e_3$, or adjacent to no other vertex by a top edge. In this case, pair $e_1$ with $e_2$.
    \end{itemize}
    It is straightforward to check that the pairing given above provides an example of a pairing we claimed to exist. Moreover, given the structure of the bottom edges in $\mathrm{M}_\mathfrak{n}(\mathfrak{s})$, it follows that each top edge belonging to $\mathrm{E}_1(\mathfrak{s})$ must get paired with a bottom edge of weight two. Consequently, this establishes that the sums of the weights of edges in $\mathrm{M}_\mathfrak{n}(\mathfrak{s})$ is equal to $|\mathrm{E}_1(\mathfrak{s})|$. The result follows.
\end{proof}

\begin{corollary}
    Let $\mathfrak{s}$ be a type-A seaweed algebra for which $|\mathrm{Cen}(\mathfrak{s})|=1$ and either $\mathfrak{s}=\mathfrak{p}^A\frac{a_1|\hdots|a_n}{b_1|\hdots|b_m}$ or $\mathfrak{p}^A\frac{a_1|\hdots|a_n}{N}$ where $a_i,b_j\in [2]$ for $i\in [n]$ and $j\in [m]$. Then $\ind\mathfrak{n}(\mathfrak{s})=|\mathrm{E}_1(\mathfrak{s})|.$
\end{corollary}

It would be interesting to find a nice characterization of those type-A seaweed algebras $\mathfrak{s}$ for which the index of the nilradical is given by $|\mathrm{E}_1(\mathfrak{s})|$. More generally, it would also be interesting if one could find a combinatorial formula depending only on $\mathrm{M}(\mathfrak{s})$ for the computation of the index of the nilradical of a (type-A) seaweed algebra $\mathfrak{s}$.

\section{Future Directions}\label{sec:fd}
 
%%%%%%%%%%%%%%%%%%%%%%%%%%%%%%%%%%%%%%

\subsection{Beyond type $A$}

In type $A$, we proved that the nilradical of a seaweed subalgebra admits the decomposition $$\mathfrak n(\mathfrak s) \cong Z(\mathfrak s) \oplus \mathfrak g^\prec(P_{\mathfrak s}),$$ where $\mathfrak g^\prec(P_{\mathfrak s})$ is a nilpotent Lie poset algebra determined combinatorially by the modified meander $\mathrm{M}_{\mathfrak p}(\mathfrak s)$. 

Since Lie poset algebras have been defined in the other classical types \cite{tBCD}, it is natural
to ask whether the splitting displayed above persists for seaweed subalgebras beyond type~$A$. Extensive low-rank calculations strongly suggest that this is indeed the case.  At present, however, extending the approach developed here to the remaining classical types would require index formulas for the associated nilpotent Lie poset algebras, and no such formulas are presently known beyond type~$A$.

%%%%%%%%%%%%%%%%%%%%%%%%%%%%%%%%%%%%%%

\subsection{Another invariant: breadth}

In this paper we have established combinatorial methods for computing the index of the nilradical of (type-$A$) seaweed algebras.  Another invariant of interest in the study of nilpotent Lie algebras is the \emph{breadth} \cite{breadthposet,breadth1,breadth4,breadth3,breadth2}.  The \demph{breadth} of a Lie algebra $\mathfrak{g}$ is defined by
\[
\mathrm{b}(\mathfrak{g})=\max_{x\in\mathfrak{g}}\operatorname{rank}(\operatorname{ad}_x).
\]
In general one has the bound
\begin{equation}\label{eq:1}
    \mathrm{b}(\mathfrak{g})\le \dim[\mathfrak{g},\mathfrak{g}]
\end{equation}
since $\operatorname{im}(\operatorname{ad}_x)\subseteq[\mathfrak{g},\mathfrak{g}]$ for all $x\in\mathfrak{g}$, where $[\mathfrak{g},\mathfrak{g}]:=\operatorname{span}\{[x,y]\mid x,y\in\mathfrak{g}\}$ denotes the \demph{derived algebra} of $\mathfrak{g}$.

As a simple application of~\eqref{eq:1}, one finds that if $\mathfrak{s}$ is a seaweed
subalgebra of $\mathfrak{gl}(N,\mathbb{C})$, then
\[
\mathrm{b}(\mathfrak{s})=\dim\mathfrak{s}-N;
\]
while if $\mathfrak{s}$ is a seaweed subalgebra of $\mathfrak{sl}(N,\mathbb{C})$, then one has
\[
\mathrm{b}(\mathfrak{s})=\dim\mathfrak{s}-N+1.
\]
For the breadth of the nilradical $\mathfrak{n}(\mathfrak{s})$ of a (type-$A$) seaweed
algebra, however, no general formula is currently known.  Although partial information
can be extracted from the breadth formulas for nilpotent Lie poset algebras found in \cite{breadthposet} together with the structural results of Section~\ref{sec:posetalg}, a complete combinatorial description remains open. This leads naturally to the following questions.

\begin{que*}
Given a (type-$A$) seaweed algebra $\mathfrak{s}$, does there exist a meander-theoretic procedure for computing the breadth $\mathrm{b}(\mathfrak{n}(\mathfrak{s}))$?
\end{que*}

\begin{que*}
Let $\mathfrak{s}$ be a (type-$A$) seaweed algebra.  Do there exist combinatorial conditions characterizing when the breadth of $\mathfrak{n}(\mathfrak{s})$ attains the upper bound in~\eqref{eq:1}; that is, $$\mathrm{b}(\mathfrak{n}(\mathfrak{s}))=
\dim[\mathfrak{n}(\mathfrak{s}),\mathfrak{n}(\mathfrak{s})]?$$ More specifically, can such conditions be expressed in terms of the associated poset $P_{\mathfrak{s}}$ or one of the meanders
$\mathrm{M}(\mathfrak{s})$, $\mathrm{M}_{\mathfrak{p}}(\mathfrak{s})$, or $\mathrm{M}_{\mathfrak{n}}(\mathfrak{s})$?
\end{que*}

%%%%%%%%%%%%%%%%%%%%%%%%%%%%%%%%

\bibliographystyle{abbrv}
\bibliography{seaweed}

\end{document}